\documentclass[12pt,english]{amsart}
\usepackage{fullpage}
\input xy
\xyoption{all}

\usepackage[latin1]{inputenc}
\usepackage{amsfonts,amsmath,amssymb,amsthm,mathrsfs,enumerate}

\newcommand{\Cc}{\mathbb{C}} 
\newcommand{\Pp}{\mathbb{P}}

\newcommand{\Zz}{\mathbb{Z}}
\newcommand{\Qq}{\mathbb{Q}} 
\newcommand{\Ff}{\mathbb{F}}

{\theoremstyle{plain}
\newtheorem{theorem}{Theorem}[section]    

\newtheorem{twisting lemma}[theorem]{Twisting lemma}

\newtheorem{lemma}[theorem]{Lemma}       
\newtheorem{proposition}[theorem]{Proposition}  
\newtheorem{corollary}[theorem]{Corollary}   

}
{\theoremstyle{definition}
\newtheorem{definition}[theorem]{Definition}      
\newtheorem{remark}[theorem]{Remark}   
\newtheorem{example}[theorem]{Example}

\newtheorem{BB problem}[theorem]{Pre-BB problem}
\newtheorem{question}[theorem]{Question}
\newtheorem{notation}[theorem]{Notation}
}

\font\sevenrm=cmr10 scaled 700

\def\sep{{\scriptsize\hbox{\rm sep}}}

\usepackage{enumitem}

\def\Aut{\hbox{\rm Aut}}

\def\Spec{\hbox{\rm Spec}}
\def\Gabs{\hbox{\rm Gal}}

\def\Gal{\hbox{\rm Gal}}

\def\Cen{\hbox{\rm Cen}}
\def\Inn{\hbox{\rm Inn}}

\def\Nor{\hbox{\rm Nor}}

\def\cd{\hbox{\rm cd}}

\def\Out{\hbox{\rm Out}}
\newcommand{\LL}{\mathbb{L}}
\newcommand{\GG}{\mathbb{G}}
\def\gp{\Gamma}

\def\cm{\hbox{\hbox{\rm C}\kern-5pt{\raise 1pt\hbox{$|$}}}}

\def\lhfl#1#2{\smash{\mathop{\hbox to 12mm{\leftarrowfill}}
\limits^{#1}_{#2}}}

\def\rhfl#1#2{\smash{\mathop{\hbox to 12mm{\rightarrowfill}}
\limits^{#1}_{#2}}}

\def\build#1_#2^#3{\mathrel{
\mathop{\kern 0pt#1}\limits_{#2}^{#3}}}

\def\htrait#1#2{\smash{\mathop{\hbox to 12mm{\hrulefill}}
\limits^{#1}_{#2}}}

\def\sxbullet{{\raise 2pt\hbox{\bf .}}}

\begin{document}

\title[Pre-Galois Theory]{Pre-Galois Theory}

\author{Pierre D\`ebes}

\author{David Harbater}

\email{Pierre.Debes@univ-lille.fr}

\email{harbater@math.upenn.edu}

\address{Laboratoire Paul Painlev\'e, Math\'ematiques, Universit\'e de Lille, 59655 Villeneuve d'Ascq Cedex, France}

\address{Department of Mathematics, University of Pennsylvania, Philadelphia, PA 19104-6395, USA. }



\date{November 4, 2019}

\thanks{The first author was supported in part by the Labex CEMPI  (ANR-11-LABX-0007-01).  The second author was supported in part by NSF grants DMS-1463733 and DMS-1805439.}

\begin{abstract} 
We introduce and study a class of field extensions that we call {\it pre-Galois}; viz.\ extensions that become Galois after some linearly disjoint Galois base change $L/k$.  Among them are {\it geometrically Galois} extensions of $\kappa(T)$, with $\kappa$ a field: extensions that become Galois and remain of the same degree over $\overline \kappa(T)$. We develop a pre-Galois theory that includes a Galois correspondence, and investigate the corresponding variants of the inverse Galois problem. We provide answers in situations where the classical analogs are not known. In particular, for every finite simple group $G$, some power $G^n$ is a geometric Galois group over $k$, and is a pre-Galois group over $k$ if $k$ is Hilbertian. For every finite group $G$, the same conclusion holds for $G$ itself ($n=1$) if $k=\Qq^{\rm ab}$ and $G$ has a weakly rigid tuple of conjugacy classes; and then $G$ is a regular  Galois group over an extension of $\Qq^{\rm ab}$ of degree dividing the order of $\Out(G)$.  
We also show that the inverse problem for pre-Galois extensions over a field $k$ (that every finite group is a pre-Galois group over $k$) is equivalent to the {\it a priori} stronger inverse Galois problem over $k$, and similarly for the geometric vs.\ regular variants.
\end{abstract}

\maketitle


\section{Introduction} \label{sec:intro}

Given a field $k$, the {\em inverse Galois problem} over $k$ asks:

\smallskip

\centerline{(IGP/$k$): Is every finite group a Galois group over $k$?}

\smallskip

The answer is negative for many fields (e.g., algebraically closed fields, finite fields, and local fields), though the answer is believed to be affirmative for global fields, and more generally for Hilbertian fields.   This remains open, however, for global fields, and in particular for $\Qq$.  The related {\em regular inverse Galois problem} over $k$ asks:
\smallskip

\centerline{(RIGP/$k$): Is every finite group a regular Galois group over $k$?}

\smallskip

\noindent
 Recall that a ``regular Galois group over $k$'' is the Galois group of a Galois field extension $F$ of $k(T)$ that is $k$-regular; i.e., in which $k$ is algebraically closed.

The latter question is conjectured to have an affirmative answer for {\em all} fields; and if this is the case for a Hilbertian field $k$, then the question (IGP/$k$) also has an affirmative answer over $k$.  In fact, most known realizations of simple groups as Galois groups over $\Qq$ have been obtained by finding a regular realization of that group over $\Qq$, typically by the method of rigidity.  An affirmative answer to RIGP is known for fields that are algebraically closed (\cite{Ha1}, Corollary~1.5), for complete discretely valued fields (\cite{Ha2}, Theorem~2.3, Corollary~2.4), for the field of totally real algebraic numbers \cite{DeFr94}, its $p$-adic analogs \cite{De95}, and more generally for the class of fields that Pop introduced and called ``large'' (see \cite{PopLarge}; these fields are also called ``ample'').  But in general the problem remains wide open.

In this paper we raise several related but more accessible questions,
concerning extensions that become Galois after a base change.

We call a finite field extension $F/k(T)$ {\em geometrically Galois} if $F \otimes_k \overline k $ is a Galois field 
extension of $\overline k(T)$;
and in that case we say that $\Gal(F \otimes_k \overline k/\overline k(T))$
is a {\em geometric Galois group} over $k$.  
Note that $F/k(T)$ is necessarily $k$-regular and separable.
The {\em geometric inverse Galois problem} asks the following, which is weaker than RIGP/$k$:

\smallskip

\centerline{(Geo-IGP/$k$): Is every finite group a geometric Galois group over $k$?}

\smallskip

Observe that if $F/k(T)$ is geometrically Galois, then 
there is a finite Galois extension $L/k$ such that $F \otimes_k L$ is 
an $L$-{\em regular} Galois field extension of $L(T)$.
This suggests:

\begin{question} \label{geom_Gal_base_chg}
If $G$ is a geometric Galois group over $k$, for how small an extension $L/k$ is $G$ a regular Galois group over $L$?
\end{question}

More generally, a field extension $E/k$ will be called {\em potentially Galois} with {\em Galois group} $G$ if there is a finite field extension $L/k$ such that 

\smallskip

\centerline{$(*)$ \ \ $E \otimes_k L$ is a Galois field extension of $L$ with Galois group $G$.}  

\smallskip

\noindent If $L$ can be chosen to be Galois over $k$, we call $E/k$ {\em pre-Galois}.  In the above situations, we say that $G$ is a {\em potential Galois group} (resp.\ {\em pre-Galois group}) over $k$, and we also say that the extension $E/k$ is potentially Galois (resp.\ pre-Galois) {\em over} $L$.
For example the extension $\Qq(\root{3} \of {2})/\Qq$ is pre-Galois with group $C_3$, by taking $L = \Qq(\zeta_3)$.

The following question, the {\it pre-inverse Galois problem} over $k$, naturally arises:
\smallskip

\centerline{(Pre-IGP/$k$): Is every finite group a pre-Galois group over $k$?}

\smallskip
Every geometrically Galois extension of $k(T)$ is pre-Galois over $k(T)$ (Proposition \ref{prop:geom vs pot}(\ref{geom Gal pre})). Also note that every pre-Galois extension of a field $k$ is potentially Galois over $k$; and every potentially Galois extension of a field $k$ is separable over $k$.

\smallskip
Our questions and answers concerning these notions are of three types.

\subsection{Pre-inverse Galois problems} \label{subsec:pre-inverse}
We give partial answers to the above problems, Geo-IGP and Pre-IGP.  We first show that over a Hilbertian field, every finite group is a potential Galois group (Proposition~\ref{hilb flds pot Gal cond}). Regarding Geo-IGP, we show the following 
over an arbitrary field $k$ (see Corollary \ref{cor:power-simple}):

\noindent
(a) {every finite group is the quotient of a geometric Galois group}, and 

\noindent
(b) {if $G$ is a simple group then some power $G^n$ is a geometric Galois group}. 

\noindent
Using Proposition \ref{intro:IGP-implications}, if $k$ is Hilbertian then one can deduce the corresponding assertions (a) and (b)
for Pre-IGP; i.e., with ``pre-Galois group'' replacing ``geometric Galois group'' (see Corollary \ref{cor:power-simple}). 

Under the assumption that the exact sequence $1 \to Z(G) \to G \to \Inn(G) \to 1$ is split, Corollary~\ref{preGal Gal special gps} shows that, over an arbitrary field, if $G$ is a pre-Galois group, then it is a Galois group.  Under this same assumption (or under the alternative hypothesis that $\cd(k) \le 1$), Corollary~\ref{thm:fod}(\ref{geom gp bound}) shows that, over a field $k$, if $G$ is a geometric Galois 
group with $\Out(G)$ trivial, then it is a regular Galois group.  
Even without the above assumptions, 
Corollary~\ref{thm:fod}(\ref{geom gp wk bound}, \ref{geom gp bound}) answers Question~\ref{geom_Gal_base_chg}, by providing an explicit bound on $[L:k]$.  

This has the following consequence, observed to us by J. K\"onig:
if (Pre-IGP/$k$) holds (for all finite groups) then (IGP/$k$) holds (for all finite groups); thus the two problems (Pre-IGP/$k$) and (IGP/$k$) are equivalent (Corollary~\ref{preIGP=IGP}). Similarly, the two problems (Geo-IGP/$k$) and  (RIGP/$k$) (for all finite groups) are equivalent (see Remark~\ref{GeoIGP=RIGP}).

In the case that $k = \Qq^{\rm ab}$, and $G$ is a finite group with a weakly rigid tuple of conjugacy classes (see Section~\ref{ssec:ext-fom}), we show in Corollary~\ref{cor: wk rig}
that Geo-IGP and Pre-IGP have affirmative answers for $G$ over $k$, along with a weak form of RIGP and IGP.  A generalization appears at Theorem~\ref{thm:main wk rig}.

\subsection{Pre-Galois theory} \label{subsec:pre-Gal}

It is easily seen that pre-Galois extensions are not always Galois: e.g., $\Qq(\root{3} \of {2})/\Qq$ is 
pre-Galois but not Galois. Similary $\Qq(T^{1/n})/\Qq(T)$ is geometrically Galois but not Galois. 
We consider these other questions, about how the new Galois properties compare to each other:

\begin{question} \label{always_pp}
Let $k$ be any field.
\renewcommand{\theenumi}{\alph{enumi}}
\renewcommand{\labelenumi}{(\alph{enumi})}
\begin{enumerate}
\item \label{always_pot} 
Is every finite separable field extension of $k$ potentially Galois?
\item \label{pot_pre}
Is every potentially Galois extension of $k$ pre-Galois?
\item \label{preGal geomGal}
If every $k$-regular pre-Galois extension of $k(T)$ geometrically Galois? 
\end{enumerate}
\end{question}

We show that the answers to parts (\ref{always_pot}), (\ref{pot_pre}), and (\ref{preGal geomGal}) of Question~\ref{always_pp} are ``no'' in general, 
even for Hilbertian fields (see Propositions~\ref{prop:simple_2d},  \ref{hilb flds pot Gal cond}(\ref{potentially but no pre})
and  \ref{prop:geom vs pot}(\ref{pre not geom})). 
\smallskip

We also consider the following questions, which contribute to a ``pre-Galois theory'': 

\begin{question} \label{inv_ppGal}
Let $E/k$ be any finite separable field extension.
\renewcommand{\theenumi}{\alph{enumi}}
\renewcommand{\labelenumi}{(\alph{enumi})}
\begin {enumerate}
\item \label{unique_gp}
If $E/k$ is pre-Galois, is the pre-Galois group unique?
\item \label{minimal}
What are the minimal field extensions (resp.\
minimal Galois field extensions) $L_0/k$ 
such that $E \otimes_k L_0$ is a Galois field extension of $L_0$?
Does {\em every} field extension $L/k$ satisfying condition $(*)$ contain a unique such minimal extension?
\item \label{pot Gal cor q}
If $E/k$ is potentially Galois with group $G$, is there an analog of the usual Galois correspondence that relates the sub-extensions of $E/k$ to subgroups of $G$?
\end{enumerate}
\end{question}

We show that the answer to Question~\ref{inv_ppGal}(\ref{unique_gp}) is generally ``no'' (see Example~\ref{ex:two-normal-complements}).  In fact, it is even possible for an extension to be pre-Galois with respect to one group and to be potentially Galois but not pre-Galois with respect to a different group (see Example~\ref{preGal and potGal ex}); and for an extension of $k(T)$ to be geometrically Galois with respect to one group and to be pre-Galois but not geometrically Galois with respect to a different group (see Example~\ref{geom also pre other gp}).  But we show that the answer to Question~\ref{inv_ppGal}(\ref{unique_gp}) is ``yes'', i.e.\ there is a unique pre-Galois group, if the extension has a pre-Galois group that is simple (see Proposition~\ref{prop:simple}).  
Theorem~\ref{prop:potentiallyGoG_characterization} gives an explicit answer to Question~\ref{inv_ppGal}(\ref{minimal}). 
An answer to Question~\ref{inv_ppGal}(\ref{pot Gal cor q})
is given in Proposition~\ref{prop:Galois_correspondence}.
\smallskip

\subsection{Lifting problems} \label{subsec: lifting}

An open question in Galois theory (called the arithmetic lifting problem or the Beckmann-Black problem) asks whether for every finite Galois extension $E/k$ there is a regular Galois extension $F/k(T)$ such that $E/k$ is obtained by specializing $T$ to some element of $k$.  This is known to hold in some cases (e.g., \cite{Beckmann}, \cite{elena-black1}), and no counterexamples are known.  It was shown in \cite[Proposition 1.2]{DeBB} that 
if this conjecture holds for all $k$, then RIGP also holds over every field.  
A more accessible question is this:

\begin{question} \label{geomBB}
Given a finite group $G$ and a Galois field extension $E/k$ of group $G$, is there a geometrically Galois extension $F/k(T)$ that has group $G$ and that specializes to $E/k$ at some $t_0 \in k$?
\end{question}

An affirmative answer is known for $k$ an ample (large) field such as $\Qq_p$ and $k_0((t))$: see \cite{DeBB}, \cite{MR1745009}, \cite{MR1841345} (with the difference that in the last two references, it is the original form of the Beckmann-Black problem that is solved).  
Theorem~\ref{geom BB thm} extends this affirmative answer under the weaker assumption that $G$ is the Galois group of a $k$-regular extension $F/k(T)$ such that the corresponding cover $X\rightarrow \Pp^1$ has a $k$-rational point above an unbranched point $t_0\in \Pp^1(k)$ (this assumption holds in particular if $k$ is ample; see \cite[Remark 4.3]{DeDes1}).  Theorem~\ref{geom BB thm} shows further that there is a choice of the geometrically Galois extension  $F/k(T)$ in Question~\ref{geomBB} having the additional property that the constant extension in the Galois closure agrees with the given extension $E/k$. Remark \ref{rem:constant-extension}(\ref{constant_extension}) discusses the value of this additional conclusion.

\smallskip

Following this introduction, we devote Section~\ref{sec:potentially-galois} to discussing pre-Galois theory over a field $k$.  Then in Section~\ref{sec:function-fields}, we consider the case of function fields, and in particular extensions that are geometrically Galois.

\smallskip

We thank Joachim K\"onig for helpful comments about this manuscript, especially concerning the relationship between the usual inverse problems in Galois theory and the ones that we consider here. We also thank Bob Guralnick for providing several group-theoretical examples and counterexamples to us.


\section{Pre-Galois extensions} \label{sec:potentially-galois}

Section \ref{subsec: struc pre} introduces and investigates the notions of potentially Galois and pre-Galois extensions. Some first questions from Section \ref{sec:intro} are answered in this first subsection. More are answered in Section \ref{subsec:ex} which gives further examples and counterexamples. The analog for potentially Galois extensions of the classical Galois correspondence appears in Section \ref{pot Gal corresp}. Finally Section \ref{subsec: Hopf} compares our pre-Galois theory with the previously introduced Hopf Galois theory.

\subsection{Structure of pre-Galois extensions} \label{subsec: struc pre}

Theorem  \ref{prop:potentiallyGoG_characterization} and Corollary \ref{prop:potentiallyGoG_characterization_cor} are the main structural conclusions of this subsection. Corollaries \ref{corollary:potential_Galois_over_Galois_group}, \ref{preGal Gal special gps}, \ref{preIGP=IGP} provide general implications towards the pre-inverse Galois problem. 
\vskip 1mm

Let $E/k$ be a finite extension of fields.  Given an overfield  $L$ of $k$ (not necessarily algebraic), and a finite group $G$, we say that $E/k$ is {\em potentially Galois over $L$ with group} $G$ if $E \otimes_k L$ is a Galois field extension of $L$ of Galois group $G$; and in this situation, 
if $L/k$ is a (not necessarily finite) Galois field extension, we say that $E/k$ is {\em pre-Galois over $L$ with group} $G$.  
Note that if $E/k$ is potentially Galois over $L$ then $E/k$ is necessarily separable, since the base change $E \otimes_k L$ is separable over $L$.  So we will restrict attention to 
finite separable field extensions $E/k$.   There is thus a Galois closure $\widehat E$ of $E/k$.

Given any overfield $L/k$, we may embed $E$ into a separable closure $L^{\rm sep}$ of $L$ as a $k$-algebra, and so we may take the compositum of fields $EL$ in $L^{\rm sep}$.  The extension $E/k$ is then potentially Galois over $L$ with group $G$ if and only if 
the field extension $EL/L$ is Galois with group $G$ (and so this does not depend on the choice of $k$-embedding $E \hookrightarrow L^{\rm sep}$).  

Also, $E/k$ is potentially Galois over $L$ if and only if $EL/L$ is Galois of degree equal to $[E:k]$; this equality is equivalent to $E$ and $L$ being linearly disjoint over $k$.
As another equivalent condition, $E/k$ is potentially Galois over $L$ if and only if the automorphism group ${\rm Aut}(EL/L)$ is of order equal to $[E:k]$.

\begin{lemma} \label{la:pp_conditions}
Let $E/k$ be a finite field extension, and let $L$ be an overfield of $k$.
\renewcommand{\theenumi}{\alph{enumi}}
\renewcommand{\labelenumi}{(\alph{enumi})}
\begin{enumerate}
\item \label{sep}
If $E/k$ is potentially Galois over $L$ then the compositum $EL$ contains the Galois closure $\widehat E$ of $E$ over $k$.
\item \label{contains_Gal_cl} 
If $E/k$ is pre-Galois over $L$ then $EL/k$ is Galois.
\end{enumerate}
\end{lemma}

\begin{proof}
For part~(\ref{sep}), the extension $E/k$ is separable and so has a primitive element $y_1$.  Let $p(Y)=\prod_{i=1}^m (Y-y_i)\in k[Y]$ be the minimal polynomial of $y_1$ over $k$, where $y_i \in \widehat E$ and $m = [E:k]$.
So $E=k(y_1)$ and $\widehat E = k(y_1,\ldots,y_m)$.  
As $EL/L$ is Galois of degree $m$, $p(Y)$ is irreducible over $L$ and $L(y_1) = EL = L(y_1,\ldots,y_d)$, whence $EL=\widehat E L$.  Thus $EL \supset \widehat E$.

Part~(\ref{contains_Gal_cl}) then follows, using $EL=\widehat E L$, since the compositum of Galois extensions is Galois.
\end{proof}

Recall that a subgroup $U$ is a {\it complement} of another subgroup $V$ in a group $\Gamma$ if every element $g\in \Gamma$ can be uniquely written as $g=u v$ with $u\in U$ and $v\in V$. If $\Gamma$ is finite, this is equivalent to asserting that $|\Gamma| = |U|\cdot|V|$ and $U\cap V = \{1\}$.  It is also equivalent to $U$ acting freely and transitively on the set of left cosets of $\Gamma$ modulo $V$, by left multiplication. The group $\Gamma$ is then said to be the {\it Zappa-Sz\'ep product} (ZS-product for short) of $U$ and $V$. If in addition $U$ is normal in $G$, then the group $\Gamma$ is the semi-direct product of $U$ and $V$ with $V$ acting on $U$ by conjugation. 

\begin{proposition} \label{ZS_equiv}
Let $E/k$ be a finite separable field extension.
Let $N/k$ be a Galois extension such that $E\subset N$, and let $H$ be an arbitrary subgroup of $\Gal(N/k)$.  Then the following 
conditions are equivalent:
\renewcommand{\theenumi}{\roman{enumi}}
\renewcommand{\labelenumi}{(\roman{enumi})}
\begin{enumerate}
\item \label{pot_Gal_over_ff}
$E/k$ is potentially Galois over the fixed field $N^H$ of $H$ in $N$
 and $N=E N^H$.
\item \label{Gal_ZS}
$\Gal(N/k)$ is the ZS-product of $H$ and $\Gal(N/E)$.
\end{enumerate}
\end{proposition}

\begin{proof}
(\ref{pot_Gal_over_ff}) $\Rightarrow$ (\ref{Gal_ZS}): Assume that (\ref{pot_Gal_over_ff}) holds. 
By $E N^H = N$, together with $E=N^{\hbox{\sevenrm Gal}(N/E)}$, it follows that $H\cap \Gal(N/E)=\{1\}$.  We have
$$|H| = [N:N^H] = [EN^H: N^H]=[E:k].$$
\noindent
The equality $|\Gal(N/k)|= |\Gal(N/E)| \cdot |H|$ follows.
\vskip 1mm

(\ref{Gal_ZS}) $\Rightarrow$ (\ref{pot_Gal_over_ff}): Assume that (\ref{Gal_ZS}) holds. We have in particular 
$$|H| = |\Gal(N/k)|/|\Gal(N/E)| = [E:k].$$

\noindent
Now $E=N^{\hbox{\sevenrm Gal}(N/E)}$ and $\Gal(N/E) \cap H = \{1\}$; hence $N=EN^H$.  Thus $EN^H/N^H$ is Galois with group $H$, and its degree is equal to $[E:k]$.  
So $E/k$ is potentially Galois over $N^H$.
\end{proof}

\begin{proposition} \label{prop:potentially_characterization}
Let $k,E,N$ be as in Proposition~\ref{ZS_equiv}, and let $L$ be an overfield of $k$ such that $N \subset EL$.
Then the restriction map ${\rm res}: \Aut(EL/L)\rightarrow \Gal(N/k)$ identifies the group $\Aut(EL/L)$ to a subgroup $H$ of $\Gal(N/k)$.  Moreover $E/k$ is potentially Galois over $L$ if and only if the equivalent conditions of Proposition~\ref{ZS_equiv} hold.
\end{proposition}

\begin{proof}
The inclusions $E\subset N \subset EL$ imply that $NL=EL$, and so the
restriction map ${\rm res}:\Aut(EL/L)\rightarrow \Gal(N/k)$ is injective.  This proves the first assertion.

Assume that $E/k$ is potentially Galois over $L$.
Then $EL/L$ is Galois and $H=r(\Gal(EL/L))$ is of order $[E:k]$. 
Set $L_0=N^H$. 
The extensions $L_0/k$ and $E/k$ are linearly disjoint since $L/k$ and $E/k$ are linearly disjoint and since $L_0=N^{H}\subset (EL)^{{\rm Gal}(EL/L)}=L$. Consequently $[EL_0: L_0] = [E:k] = |H|$.
But $EL_0 \subset N$ and $[N:L_0]=|H|$.  Therefore $E L_0= N$ and $EL_0/L_0$ is Galois (of group $H$). 
That is, condition~(\ref{pot_Gal_over_ff}) of Proposition~\ref{ZS_equiv} holds.

Conversely, assume that condition~(\ref{pot_Gal_over_ff}) of Proposition~\ref{ZS_equiv} holds.
Then the group $\Aut(EL/L)$, being isomorphic to $H =\Gal(N/N^H) =\Gal(EN^H/N^H)$, has order $[EN^H: N^H]=[E:k]$.  
By the last equivalence stated prior to Lemma~\ref{la:pp_conditions},
it follows that $E/k$ is potentially Galois over $L$. 
\end{proof}

\begin{notation} \label{rem:G(L)}
Let $E/k$ be a finite separable extension, with Galois closure $\widehat E/k$. The sets ${\mathscr L}$, ${\mathscr G}$ and the maps $\LL$, $\GG$ that are defined below depend on the extension $E/k$. For simplicity we omit the reference to $E/k$ in the notation.

Write ${\mathscr L}$ for the (possibly empty) set of all overfields $L$ of $k$ such that $E/k$ is potentially Galois over $L$, and $\mathscr G$ for the (possibly empty) set of complements of 
$\Gal(\widehat E/E)$ in $\Gal(\widehat E/k)$. Note that if $G\in \mathscr G$, then $|G|=[E:k]$.

If $L\in {\mathscr L}$, then for $N=\widehat E$, both conditions $E\subset N$ of Proposition \ref{ZS_equiv} and $N\subset EL$ of Proposition \ref{prop:potentially_characterization} are satisfied (via Lemma~\ref{la:pp_conditions}(\ref{contains_Gal_cl}) for the latter). It follows from those propositions that the group $\GG(L):= {\rm res}(\Gal(EL/L))$, with ${\rm res}:\Gal(EL/L) \rightarrow \Gal(\widehat E/k)$ the restriction map, is a 
complement of $\Gal(\widehat E/E)$ in $\Gal(\widehat E/k)$; i.e. is in the set $\mathscr G$.  We thus have a map $\GG:\mathscr L \rightarrow \mathscr G$ given by 
\[L\mapsto \GG(L)= {\rm res}(\Gal(EL/L)) \subset \Gal(\widehat E/k).\]
By definition $E/k$ is potentially Galois over $L$ with group $\GG(L)$.
\end{notation}

\begin{theorem} \label{prop:potentiallyGoG_characterization} 
Let $E/k$ be a separable field extension of degree $d$, and let $\widehat E/k$ be its Galois closure.  Let $\mathscr L_{\rm m}\subset \mathscr L$ be the subset of minimal fields $L$ in $\mathscr L$. 
\renewcommand{\theenumi}{\alph{enumi}}
\renewcommand{\labelenumi}{(\alph{enumi})}
\begin{enumerate} 
\item \label{pot Gal bijection} 
The restriction $\GG: \mathscr L_{\rm m} \rightarrow \mathscr G$ of $\GG$ to $\mathscr L_{\rm m}$ is a bijection whose inverse is the map $\LL: \mathscr G \rightarrow \mathscr L_{\rm m}$
defined by $\LL(G) = \widehat E^G$.
\item \label{min fields give pot Gal}
For each $L \in \mathscr L_{\rm m}$, we have $L\subset \widehat E$, $EL=\widehat E$ and the extension $E/k$ is potentially Galois over $L$ with group $\GG(L)=\Gal(\widehat E/L)$.
\item \label{min subextens}  
For every $L\in \mathscr L$, there is a unique $L_0 \in \mathscr L_{\rm m}$ such that $L_0 \subset L$; and $L_0=\LL(\GG(L))$.  The extension $EL/L$ is obtained from $EL_0/L_0$ by base change. Furthermore, $\GG(L) = \GG(L_0)$.
Finally if $L/k$ is Galois, so is $L_0/k$, and $\GG(L)$ is normal in $\Gal(\widehat E/k)$.
\end{enumerate}
\end{theorem}

\begin{proof} The proof proceeds in several steps.

{\it First argument}. Let $L\in \mathscr L$ and
$G=\GG(L)$. It follows from the definition of $\GG(L)$ (Notation \ref{rem:G(L)}) that
$L=(EL)^{{\rm Gal}(EL/L)} \supset \widehat E^G$. As $\widehat E L = EL$ (Lemma \ref{la:pp_conditions}(\ref{contains_Gal_cl})), we obtain that
\vskip 1mm

\noindent
(*) the extension $EL/L$ is obtained from $\widehat E/\widehat E^G$ by base change $L/\widehat E^G$. 
\vskip 1mm

\noindent
Furthermore, by definition of $\mathscr G$, for every $G\in \mathscr G$, 
$\Gal(\widehat E/k)$ is the ZS-product of $G$ and $\Gal(\widehat E/E)$. It follows from Proposition \ref{ZS_equiv} ((\ref{Gal_ZS}) $\Rightarrow$ (\ref{pot_Gal_over_ff}), with $N=\widehat E$) that 
\vskip 1mm

\noindent
(**) $E/k$ is potentially Galois over $\widehat E^G$ and $\widehat E=E \widehat E^G$; in particular $\GG(\widehat E^G)= G$.
\vskip 1mm

{\it Proof that $\LL(\mathscr G)\subset \mathscr L_{\rm m}$ and $\mathscr G =\GG(\mathscr L_{\rm m})$}.
Let $G \in \mathscr G$. Then $\widehat E^G\in \mathscr L$ (from (**) above).
Now suppose that $L\in \mathscr L$ satisfies $L\subset \widehat E^G$. We have $EL\subset \widehat E$, and as we already know that $EL\supset \widehat E$, we obtain $EL=\widehat E$. It follows that $\GG(L)=\Gal(\widehat E/L) \supset G$ and this containment is in fact an equality: $\GG(L) = G$, as the groups have the same order $d$. The first argument then applies and gives $L\supset \widehat E^G$. Hence $L = \widehat E^G$ and $\widehat E^G\in \mathscr L_{\rm m}$. As this holds for every $G \in \mathscr G$, 
we have $\LL(\mathscr G)\subset \mathscr L_{\rm m}$. Also $G = \GG(\widehat E^G)$.  So since $\widehat E^G\in \mathscr L_{\rm m}$ for every $G \in \mathscr G$, it follows that $\mathscr G \subset \GG(\mathscr L_{\rm m})$.
As $\GG(\mathscr L_{\rm m}) \subset \GG(\mathscr L) \subset \mathscr G$ (Notation \ref{rem:G(L)}), we obtain $\mathscr G =\GG(\mathscr L_{\rm m})$.

\vskip 1,5mm

{\it Proof of (\ref{min fields give pot Gal}) and of $\mathscr L_{\rm m}= \LL(\mathscr G)$}. For every $L\in \mathscr L_{\rm m}$, the first argument shows that  $L \supset \widehat E^G$ for some $G\in \mathscr G$, and so $L=\widehat E^G$ since $L$ and $\widehat E^G$ are in ${\mathscr L}_{\rm m}$. Taking into account (**) above, statement (\ref{min fields give pot Gal}) follows immediately. It also follows that  $\mathscr L_{\rm m}\subset \LL(\mathscr G)$; but $\LL(\mathscr G)\subset \mathscr L_{\rm m}$, and so $\mathscr L_{\rm m}= \LL(\mathscr G)$. 

\vskip 1,5mm

{\it Proof of (\ref{pot Gal bijection})}. We already know that the maps $\LL:\mathscr G \to \mathscr L_{\rm m}$ and $\GG:\mathscr L_{\rm m} \to \mathscr G$ are well-defined and surjective. For every $L\in \mathscr L_m$, the equality
$\GG(L)=\Gal(\widehat E/L)$ (proved in (\ref{min fields give pot Gal})) yields $\LL (\GG (L)) = L$. 
Finally we know from above that for every $G\in \mathscr G$, $\LL(G)=\widehat E^G \in \mathscr L_{\rm m}$. 
Statement (\ref{min fields give pot Gal}) then gives  $\GG(\LL(G))=\Gal(\widehat E/\LL(G))=G$, completing the proof of (\ref{pot Gal bijection}).

\vskip 1,5mm

{\it Proof of (\ref{min subextens})}. Let $L\in \mathscr L$. The fact that the field $L_0= \LL(\GG(L))$ satisfies $L_0 \in \mathscr L_{\rm m}$ and $L_0 \subset L$ was already proved. Assume that there is another $L_0^\prime\in  \mathscr L_{\rm m}$ such that $L_0^\prime \subset L$. Write $L^\prime_0=\LL(G^\prime)$ with $G^\prime \in \mathscr G$.
It follows from $L_0=\widehat E^{\GG(L)} \subset L$ and $L^\prime_0=\widehat E^{G^\prime} \subset L$ that
\[\GG(L) = {\rm res}(\Gal(EL/L) ) \subset  \Gal(\widehat E/\widehat E^{\GG(L)} ) \cap  \Gal(\widehat E/\widehat E^{G^\prime}) 
= {\GG(L)} \cap G^\prime.\]
As the groups $\GG(L)$ and $G^\prime$ have the same order, 
namely $d=[E:k]$, we have necessarily $\GG(L) = G^\prime$ and so $L_0=L_0^\prime$.
The second sentence of statement (\ref{min subextens}) corresponds to (*) in the first argument. The equality $\GG(L_0)=\GG(L)$ follows from $\GG\circ \LL ={\rm Id}_{\mathscr G}$.
Finally, assume that $L/k$ is a Galois extension. Then the field $EL=\widehat E L$ is a Galois extension of $k$; and $L_0=\LL(\GG(L))$ is the fixed field in $\widehat E L$ of the subgroup ${\gp}
\subset \Gal(\widehat EL/k)$ that is
generated by $\Gal(\widehat EL/L)$ and $\Gal(\widehat EL/\widehat E)$. As the extensions 
$\widehat E/k$ and $L/k$ are Galois, both these subgroups are normal in 
$\Gal(\widehat EL/k)$. Therefore so is ${\gp}$ and hence $L_0/k$ is Galois; equivalently, 
$\GG(L)$ is normal in $\Gal(\widehat E/k)$.
\end{proof}

\begin{remark} \label{rem:post-thm} We will sometimes use the following facts, contained in Theorem \ref{prop:potentiallyGoG_characterization}:
\renewcommand{\theenumi}{\alph{enumi}}
\renewcommand{\labelenumi}{(\alph{enumi})}
\begin{enumerate}
\item \label{corresp rk a}
An overfield $L$ of $k$ is in $\mathscr L_{\rm m}$ if and only if it is of the form $L=\LL(G)$ for some unique $G\in \mathscr G$ (Theorem \ref{prop:potentiallyGoG_characterization}(\ref{pot Gal bijection})), and then 
 $\LL(G) \subset \widehat E$, $E\LL(G)=\widehat E$ and  the extension $E/k$ is potentially Galois over $\LL(G)$ with group $\GG(\LL(G))=G$ (Theorem \ref{prop:potentiallyGoG_characterization}(\ref{min fields give pot Gal})).
\item \label{corresp rk b}
If an extension $E/k$ is potentially 
or pre-Galois, then by Theorem \ref{prop:potentiallyGoG_characterization}(\ref{min fields give pot Gal})
it is so over some extension  $L/k$ that is contained in the Galois closure $\widehat E/k$, with $EL = \widehat E$: e.g.\ any $L/k$ with $L\in {\mathscr L}_{\rm m}$. (There is in fact no other choice: if $L\in \mathscr L$ and $L\subset \widehat E$, then $L\in {\mathscr L}_{\rm m}$. Indeed, from Theorem \ref{prop:potentiallyGoG_characterization}(\ref{min subextens}), if $L\in \mathscr L$, then $L$ contains some $L_0\in {\mathscr L}_{\rm m}$, which satisfies $EL_0 = \widehat E$ (Theorem \ref{prop:potentiallyGoG_characterization}(\ref{min fields give pot Gal})); and if $L\subset \widehat E$, then $EL=\widehat E$. This, combined with 
$[EL:L]=[EL_0:L_0]$, gives $L=L_0$).
\end{enumerate}
\end{remark}

\begin{corollary} \label{prop:potentiallyGoG_characterization_cor} Let $E/k$ be a degree $d$ separable extension and $\widehat E/k$ be its Galois closure.
\renewcommand{\theenumi}{\alph{enumi}}
\renewcommand{\labelenumi}{(\alph{enumi})}
\begin{enumerate}
\item\label{pot cor SZ}
$E/k$ is potentially Galois (resp.\ pre-Galois) if and only if $\Gal(\widehat E/k)$ is the ZS-product of some 
order $d$ subgroup $G$ and of $\Gal(\widehat E/E)$ (resp.\ the semi-direct product of some order $d$ normal subgroup $G$ and $\Gal(\widehat E/E)$); and 
then $E/k$ is potentially Galois (resp.\ pre-Galois) of group $G$.

\item \label{pot cor complem}
The potential Galois groups (resp.\ the pre-Galois groups) of $E/k$  are exactly the complements (resp.\ the normal complements) of $\Gal(\widehat E/E)$ in $\Gal(\widehat E/k)$.

\item \label{pot cor pre iff}
An extension $E/k$ is pre-Galois if and only 
it is potentially Galois and one of the minimal extensions $\LL(G)/k$ ($G\in \mathscr G$)
is  Galois.
\end{enumerate}

\end{corollary}

\begin{proof} (\ref{pot cor SZ}) With our notation, $E/k$ is potentially Galois if and only if $\mathscr L\not= \emptyset$, which, from Theorem \ref{prop:potentiallyGoG_characterization}(\ref{min subextens}), is equivalent to $\mathscr L_{\rm m} \not=\emptyset$, which in turn is equivalent to $\mathscr G\not=\emptyset$ by  Theorem \ref{prop:potentiallyGoG_characterization}(\ref{pot Gal bijection}). This proves the ``potentially Galois" 
part of statement (\ref{pot cor SZ}). The ``pre-Galois'' part is proved similarly using the final part of 
Theorem \ref{prop:potentiallyGoG_characterization}(\ref{min subextens}).
\vskip 1mm

\noindent
(\ref{pot cor complem}) A group $G$ being a potential Galois group (resp. pre-Galois group) means that there is an extension $E/k$ that is potentially Galois (resp. pre-Galois) over $L$ and $G$ is isomorphic to $\GG(L)$. Thus
(\ref{pot cor complem}) straightforwardly follows from (\ref{pot cor SZ}).

\vskip 1mm

\noindent
(\ref{pot cor pre iff}) The direct part is straightforward from Theorem~\ref{prop:potentiallyGoG_characterization}(\ref{min subextens}). The reverse part is clear.
\end{proof}

\begin{remark} \label{rem:inside-Galois-closure}  
\renewcommand{\theenumi}{\alph{enumi}}
\renewcommand{\labelenumi}{(\alph{enumi})}
\begin{enumerate}
\item \label{inside rk a}
By Theorem \ref{prop:potentiallyGoG_characterization}(\ref{pot Gal bijection}), if $L\in \mathscr L_{\rm m}$,
 then $L=\LL(\GG(L))$. If $E/k$ is pre-Galois over $L$ (of group $\GG(L)$), the action of $\Gal(L/k)$ on $\GG(L)$ is faithful.  To see this, note that $EL/k$ is Galois (by Lemma~\ref{la:pp_conditions}(\ref{contains_Gal_cl})),
and $\Gal(EL/k)$ is a semi-direct product of the normal subgroup $\GG(L)$ with 
the quotient group $\Gal(L/k)=\Gal(EL/E)$.  If the action of $\Gal(L/k)$ on $\GG(L)$ were not faithful,
we could take the invariant subfield in $L$ of the kernel of the action, and that would be a smaller 
(Galois) extension of $k$ whose pullback makes $E/k$ Galois.

\item \label{inside rk b}
By Theorem \ref{prop:potentiallyGoG_characterization}(\ref{min subextens}), if $L\in \mathscr L$ and $L/k$ is Galois (i.e., $E/k$ pre-Galois over $L$), then the same is true for the corresponding minimal field $L_0=\LL(\GG(L))\in \mathscr L_{\rm m}$. We note however that ``Galois'' cannot be weakened to ``pre-Galois'' in this statement; viz.,  it may happen that $L/k$ is pre-Galois but $L_0/k$ is not. We provide an example in Section~\ref{ssec:lifting_problems}  (see Example~\ref{example:L_0/k_not_pre-Galois}).
\end{enumerate}
\end{remark}

\begin{corollary} \label{corollary:potential_Galois_over_Galois_group}
A finite group $G$ is a pre-Galois group over $k$ if and only if there exists a group action $A\rightarrow {\rm Aut}(G)$ such that the semi-direct product $G\rtimes A$ is a Galois group over $k$.
\end{corollary}

\begin{proof} ($\Rightarrow$) 
If $G$ is a pre-Galois group over $k$, then $G$ is a normal complement of $\Gal(\widehat E/E)$ in 
$\Gal(\widehat E/k)$, by Corollary~\ref{prop:potentiallyGoG_characterization_cor}(\ref{pot cor complem}).  Thus, with $A:= \Gal(\widehat E/E)$, we have $\Gal(\widehat E/k) = G\rtimes A$, where $A$ acts on $G$ by conjugation in 
$\Gal(\widehat E/k)$.

\smallskip

\noindent
($\Leftarrow$) Assume $G\rtimes A$ is the group of some Galois extension $N/k$. For $E=N^A$, the extension $E/k$ satisfies condition~(\ref{Gal_ZS}) from Proposition~\ref{ZS_equiv} with $H=G$. Therefore we have condition~(\ref{pot_Gal_over_ff}) of that result; in particular $E/k$ is  potentially Galois over $L=N^{G}$. As $G$ is normal in $\Gal(N/k)$, the extension $L/k$ is Galois and $E/k$ is  pre-Galois over $L$, of group $G$.
\end{proof}

Joachim K\"onig observed the following consequences of Corollary~\ref{corollary:potential_Galois_over_Galois_group}.

\begin{corollary} \label{preGal Gal special gps}
Let $G$ be a finite group such that $\Out(G)$ is trivial, and such that the exact sequence 
$1 \to Z(G) \to G \to \Inn(G) \to 1$ is split.  If $G$ is a pre-Galois group over a field $k$, then $G$ is a Galois group over $k$. 
\end{corollary}

\begin{proof}
By Corollary~\ref{corollary:potential_Galois_over_Galois_group}, it suffices to show that every semi-direct product $G \rtimes A$ is isomorphic to $G \times A$, since then $G$ is a quotient of the Galois group $G \times A$ and hence itself a Galois group.   

Given a semi-direct product $G \rtimes A$, let $\alpha:A \to \Aut(G)$ be the associated action.
Since $\Out(G)$ is trivial, $\Aut(G)=\Inn(G)$.  Composing the resulting map $A \to \Inn(G)$ with the given section
$\Inn(G) \to G$ of $G \to \Inn(G)$, we obtain a homomorphism $\sigma:A \to G$ such that $\alpha(a)(g)=
\sigma(a)g\sigma(a)^{-1}$ for $a \in A$ and $g \in G$.

Let $A^* = \{(\sigma(a),a^{-1}) \in G \rtimes A\,|\,a \in A\}$.  One checks directly that the map $A \to A^*$ given by $a \mapsto (\sigma(a),a^{-1})$ is an isomorphism; that the subgroups $G \times 1$ and $A^*$ of $G \rtimes A$ commute; that they intersect trivially; and that they generate $G \rtimes A$.  Hence 
$G \rtimes A \cong (G \times 1) \times A^* \cong G \times A$.
\end{proof}

Recall that a group $G$ is said to be {\it complete} if both $Z(G)$ and ${\rm Out}(G)$ are trivial. Symmetric groups $S_n$ with $n\notin \{2,6\}$ and all automorphism groups of non-abelian simple groups are examples of complete groups
\cite[Section 13.5.10]{Ro96}.

\begin{corollary} \label{preIGP=IGP} Let $k$ be an arbitrary field. The following statements are equivalent:
\renewcommand{\theenumi}{\roman{enumi}}
\renewcommand{\labelenumi}{(\roman{enumi})}
\begin{enumerate}
\item \label{IGP-k}
Every finite group is a Galois group over $k$.
\item \label{preIGP-k}
Every finite group is a pre-Galois group over $k$.
\item \label{normIGP-k}
Every finite group is a normal subgroup of a Galois group over $k$.
\item \label{complete normIGP-k}
Every complete finite group is a normal subgroup of a Galois group over $k$.
\end{enumerate}
\end{corollary}

\begin{proof}
The implication (\ref{IGP-k})$\Rightarrow$(\ref{preIGP-k}) is obvious; (\ref{preIGP-k})$\Rightarrow$(\ref{normIGP-k}) follows from Corollary \ref{corollary:potential_Galois_over_Galois_group}; and (\ref{normIGP-k})$\Rightarrow$(\ref{complete normIGP-k}) is obvious. 
We are left with proving (\ref{complete normIGP-k})$\Rightarrow$(\ref{IGP-k}).

Let $G$ be a finite group. By \cite[Theorem 1]{HR80}, $G$ is a quotient of some
complete group $\widetilde G$. By (\ref{complete normIGP-k}), $\widetilde G$ is a normal subgroup of some group $\Gamma$
that is a Galois
group over $k$. Since $\widetilde G$ is complete, the exact sequence 
\[1 \rightarrow \widetilde G \rightarrow \Gamma \rightarrow \Gamma/\widetilde G \rightarrow 1\]
splits \cite[Section 13.5.8]{Ro96}
and $\Gamma$ is isomorphic to a semidirect product $\widetilde G \rtimes A$ with $A = \Gamma/\widetilde G$. From  
Corollary \ref{corollary:potential_Galois_over_Galois_group}, $\widetilde G$ is a pre-Galois group over $k$. 
But since $\widetilde G$ is complete, this implies that $\widetilde G$ is a Galois group over $k$ 
(Corollary \ref{preGal Gal special gps}). It follows that $G$ itself is a Galois group over $k$.     
\end{proof}

\subsection{Examples and counterexamples} \label{subsec:ex} This section presents in particular our answers to 
Question \ref{always_pp}(\ref{always_pot})-(\ref{pot_pre}) about the connections between the notions, and to Question 
\ref{inv_ppGal}(\ref{unique_gp}) about the uniqueness of the pre-Galois group.
\vskip 1mm

Concerning Pre-IGP, a pre-Galois group $G$ over a field $k$ must satisfy these 
conditions:

\noindent
- $G$ is a potential Galois group over $k$,

\noindent
- $G$ is a Galois group $\Gal(N/L)$ with $L/k$ a Galois extension.
\vskip 1mm

As parts (\ref{hilb pot Gal}) and (\ref{hilb Gal gps over Gal extn}) of Proposition~\ref{hilb flds pot Gal cond}  below show, both of these conditions hold for {\em all} finite groups $G$, provided that $k$ is a global field, or more generally any Hilbertian field. 

We adhere to the definition of a Hilbertian field given in \cite[Section~12.1]{FrJa}.
Given integers $r,m\geq 1$ and polynomials $f_1(T_1,\ldots,T_r,Y),\ldots, f_m(T_1,\ldots,T_r,Y)$ that are irreducible in $k(T_1,\ldots,T_r)[Y]$ and separable in $Y$, the set of all  $(t_1,\ldots,t_r) \in k^r$ such that  $f_i(t_1,\ldots,t_r,Y)$ is irreducible in $k[Y]$, $i=1,\ldots,r$, is called a {\it separable Hilbert subset} of $k^r$, and the field $k$ is said to be {\it Hilbertian} if for every $r\geq 1$, every separable Hilbert subset of $k^r$ is Zariski-dense in $k^r$. 
Classical Hilbertian fields include the field $\Qq$, the rational function fields $\Ff_q(u)$ (with $u$ some indeterminate) and all of their finitely generated extensions \cite[Theorem 13.4.2]{FrJa}.

Recall that $S_d$ is a Galois group over every Hilbertian field.  This follows from the fact that $S_d$ is a Galois group over a purely transcendental extension of any field $k$ (viz., the Galois group of $k(x_1,\dots,x_d)/k(\sigma_1,\dots,\sigma_d)$, where $\sigma_i$ is the $i$-th elementary symmetric polynomial in $x_1,\dots,x_d$).

\begin{proposition} \label{hilb flds pot Gal cond}
Let $k$ be a Hilbertian field and let $G$ be any finite group.
\renewcommand{\theenumi}{\alph{enumi}}
\renewcommand{\labelenumi}{(\alph{enumi})}
\begin{enumerate}
\item \label{hilb pot Gal}
Then $G$ is a potential Galois group over $k$, and, every extension $E/k$ of degree $d=|G|$ and with Galois closure of group $S_d$ is potentially Galois of group $G$.
\item \label{potentially but no pre}
If $d\geq 5$, there is a potentially Galois extension of $k$ of degree $d$, which 
is not pre-Galois.
\item \label{hilb Gal gps over Gal extn}
There is a finite Galois extension $L/k$ such that $G$ 
is the Galois group of some finite Galois field extension $N/L$.
\end{enumerate}
\end{proposition}

\begin{proof}
Let $d\geq 1$ be an integer and $N/k$ a Galois extension of group $S_d$; this exists by the paragraph before the proposition.  Let $M\subset S_d$ the subgroup fixing one letter; thus $M$ has index $d$ and is isomorphic to $S_{d-1}$.  

(\ref{hilb pot Gal}) Assume $N/k$ is the Galois closure $\widehat E/k$ of a given degree $d$ extension $E/k$. Then $E$ is the fixed field of  $M$ in $N$ (for some letter). Let $G$ be any group of order $d$. Embed it in $S_d$ via the regular representation. The group $G$ 
is then a complement of $M$ in $S_d$. By Corollary~\ref{prop:potentiallyGoG_characterization_cor}(\ref{pot cor SZ}), the extension 
$E/k$ is potentially Galois of group $G$.

(\ref{potentially but no pre}) Assume $d\geq 5$. Let $E$ be the fixed field of $M$ in the extension $N/k$.  Any complement of $M$ in $S_d$ has order $d$; but no such 
subgroup is normal in $S_d$.  Hence $E/k$ is not pre-Galois, by Corollary~\ref{prop:potentiallyGoG_characterization_cor}(\ref{pot cor complem}).  But the group $M$ has a (non-normal) complement in $S_d$, e.g. the group generated by a $d$-cycle in $S_d$; and so $E/k$ is potentially Galois, again by Corollary~\ref{prop:potentiallyGoG_characterization_cor}(\ref{pot cor complem}).

(\ref{hilb Gal gps over Gal extn}) With $k^{\rm sep}$ the separable closure of $k$, the group $G$ is a Galois group over $k^{\rm sep}(T)$, by the RIGP over separably closed fields (a special case of RIGP over large fields).
Hence there is a finite separable extension $L/k$ and a Galois field extension $F/L(T)$ with group $G$, such that $L$ is algebraically closed in $F$.  After replacing $L$ by its Galois closure over $k$, we may assume that $L/k$ is Galois.  Since $k$ is Hilbertian, it follows that some specialization $E/L$ of $F/L(T)$ is a Galois field extension of $L$ of group $G$.  
\end{proof}

\begin{remark} 
As the proof of Proposition~\ref{hilb flds pot Gal cond} shows, parts~(\ref{hilb pot Gal}) and~(\ref{potentially but no pre}) hold more generally for any field $k$ over which $S_d$ is a Galois group, even if $k$ is not assumed Hilbertian.
\end{remark}

Proposition \ref{hilb flds pot Gal cond}(\ref{hilb pot Gal}) solves the ``potential inverse Galois problem'' over a Hilbertian field,
and, with Proposition \ref{hilb flds pot Gal cond}(\ref{hilb Gal gps over Gal extn}), provides evidence for an affirmative answer to Pre-IGP in that situation. Note that in the case of $k = \Qq$, part~(\ref{hilb Gal gps over Gal extn})
was shown at \cite[Proposition~1.4]{Ha2}.

Proposition~\ref{hilb flds pot Gal cond}(\ref{potentially but no pre}) shows that Question~\ref{always_pp}(\ref{pot_pre}) has a 
negative answer.

Concerning Question~\ref{always_pp}(\ref{always_pot}), every separable extension of degree $2$ is Galois and so pre-Galois.  Separable extensions of degree $3$ either are Galois or have a Galois closure of group $S_3$, so they are  pre-Galois too (Proposition \ref{hilb flds pot Gal cond}(\ref{hilb pot Gal})).  
Moreover, every separable extension of degree $4$ is pre-Galois, by 
\cite[Theorem~4.6]{greither-pareigis} (where it is shown that degree $4$ separable extensions are ``almost classically Galois''; see Section~\ref{subsec: Hopf} below).
Nevertheless, extensions exist that are not potentially  Galois, and so the answer to Question~\ref{always_pp}(\ref{always_pot}) is in general negative, even for Hilbertian fields $k$:

\begin{proposition} \label{prop:simple_2d}
Let $d\geq 2$ be an integer such that some simple group of order $2d$ is a Galois group over $k$. Then there is a degree $d$ separable field extension $E/k$ that is not potentially Galois.
\end{proposition}

\begin{proof} Let $N/k$ be a Galois extension of degree $2d$ whose Galois group ${\gp}$ is simple, and let $\sigma \in {\gp}$ be an element of order $2$. As ${\gp}$ is simple, the extension $E=N^\sigma$ of $k$ is not Galois and its Galois closure is $N/k$. The group $\Gal(N/E)$ has no complement in ${\gp}$ (for it would be of index $2$). From Corollary~\ref{prop:potentiallyGoG_characterization_cor}(\ref{pot cor complem}), 
$E/k$ is not potentially Galois. 
\end{proof}

So for example, over any Hilbertian field $k$, there is a separable extension of degree $30$ that is not potentially Galois over $k$, because $A_5$ is a Galois group over $k$ (every alternating group $A_n$ is classically known to be a regular 
Galois group over $\Qq$ (see \cite[Section 4.5]{Ser92}), and so a Galois group over any Hilbertian field of characteristic $0$;
see \cite{brink} for the positive characteristic case).

\begin{remark} More generally, let $E/k$ be a non-Galois degree $d$ field extension, with Galois closure $\widehat E/k$.  Suppose that no order $d$ subgroup of $\Gal(\widehat E/k)$ has a complement in $\Gal(\widehat E/k)$.  Then $E/k$ is not potentially Galois over $k$. 
In particular, let $\Gamma$ be a group with a non-trivial subgroup $M$, and assume that $M$ has no complement in $\Gamma$ and that no nontrivial subgroup of $M$ is normal in ${\gp}$.  If ${\gp}$ is the Galois group of some extension $N/k$, the extension $N^M/k$ is not potentially Galois.  For example, we may take $\Gamma=A_4$, with $M$ of order $2$.
\end{remark}

Question~\ref{inv_ppGal}(\ref{unique_gp}) also has a negative answer, since a subgroup $U$ of a Galois group $\Gamma = \Gal(E/k)$ can have more than one normal complement up to isomorphism:

\begin{example} \label{ex:two-normal-complements}
This example was provided to us by R.~Guralnick. Let $\Gamma$ be the dihedral group of order $8$, with generators $a$ and $b$ of orders $4$ and $2$.  
Let $U$ be the $2$-cyclic subgroup generated by $b$. 
Let $V$ be the cyclic subgroup generated by $a$, and let $V^\prime$ be the Klein four subgroup generated by $a^2$ and $ab$.
The groups $V$ and $V^\prime$ are each complements to $U$ and are normal in $\Gamma$, and $V$ is not isomorphic to $V^\prime$.
\end{example}

\noindent
Since every finite group $\Gamma$ is a Galois group over some number field $k$, one can therefore obtain examples of pre-Galois extensions of $k$ with more than one pre-Galois group.  

However, Proposition~\ref{prop:simple} below does provide a uniqueness assertion in a special case.  First we need a group-theoretic lemma.

\begin{lemma} \label{lem:unique-complement}
Suppose $G$, $G^\prime$ are each normal complements of a 
subgroup $U$ in a group $\Gamma$.
\begin{enumerate}
\renewcommand{\theenumi}{\alph{enumi}}
\renewcommand{\labelenumi}{(\alph{enumi})}
\item \label{la: unique elt}
For every $g\in G$, there is a unique element $\gamma \in U$ such that $g \gamma \in G^\prime$.
\item \label{la: commutes}
Every element of $G/(G\cap G^\prime)$ commutes with every element of $G^\prime/(G\cap G^\prime)$ {\rm (}inside the group $\Gamma/(G\cap G^\prime)${\rm )},
\item \label{la: bij}
The map $\phi:G\rightarrow G^\prime$ that sends $g$ to $g\gamma$ is a bijection which satisfies: 
\begin{enumerate}
[label=(\roman*)]
\renewcommand{\theenumi}{\roman{enumi}}

\item \label{la: cocycle}
$\phi(g_1 g_2) = \phi(g_2)^{g_1} \phi(g_1)$ ($g_1,g_2\in G)$,
\item \label{la: id on int}
$\phi$ is the identity on $G\cap G^\prime$, 
\item \label{la: anti-iso}
$\phi$ induces an anti-isomorphism $G/(G\cap G^\prime) \rightarrow G^\prime/(G\cap G^\prime)$. 
\end{enumerate}
\end{enumerate}
\end{lemma}

\begin{proof} (\ref{la: unique elt}): Every element $g\in \Gamma$ is uniquely of the form $g=uv$ with $u\in G^\prime$, $v \in U$. Applying this to every $g\in G\subset \Gamma$ gives the assertion.
\vskip 2pt

(\ref{la: commutes}): Let $g\in G$ and $g^{\prime} \in G^\prime$. As both $G$ and $G^\prime$ are normal in $\Gamma$, the commutator $gg^{\prime} g^{-1} (g^\prime)^{-1}$ is in $G\cap G^\prime$, whence the assertion.

(\ref{la: bij}) If $g_i \in G$ and $\gamma_i \in U$ with $g_i \gamma_i \in G'$ for $i=1,2$, then 
$g_1^{-1}g_2 =  \gamma_1 \gamma_2^{-1} \in G\cap U = \{1\}$, using that $G,U$ are complements.  This shows that $\phi$ is injective.  Hence $\phi$ is bijective as $|G|=|G'|$.  It remains to show assertions 
\ref{la: cocycle} -- \ref{la: anti-iso} of (\ref{la: bij}).

(\ref{la: bij})\ref{la: cocycle}: Write $\phi(g_1)= g_1 \gamma_1$ and $\phi(g_2)= g_2 \gamma_2$. Then we have
$$\phi(g_2)^{g_1} \phi(g_1) = g_1 (g_2 \gamma_2) g_1^{-1} (g_1 \gamma_1) = (g_1g_2) (\gamma_2 \gamma_1).$$
As $G^\prime$ is normal in $\Gamma$, $\phi(g_2)^{g_1}\in G^\prime$. As $\gamma_2 \gamma_1 \in U$, we obtain 
$$\phi(g_2)^{g_1} \phi(g_1) = \phi(g_1 g_2).$$

(\ref{la: bij})\ref{la: id on int} is clear.

(\ref{la: bij})\ref{la: anti-iso}: By (\ref{la: bij})\ref{la: cocycle} and (\ref{la: commutes}), it follows that the map $\overline \phi: G \rightarrow G^\prime/(G\cap G^\prime)$ that sends every element $g\in G$ to the coset $\phi(g) (G\cap G^\prime)$ satisfies $\overline \phi(g_1 g_2) = \overline \phi(g_2) \overline \phi(g_1)$, {\it i.e.} is an anti-morphism. From (\ref{la: bij})\ref{la: id on int}, $G\cap G^{\prime} \subset {\rm ker}(\overline \phi)$. The other containment $G\cap G^{\prime} \supset {\rm ker}(\overline \phi)$ is clear: if $\phi(g)\in G\cap G^\prime$, there exists $\gamma\in U$ such that $g\gamma\in G\cap G^\prime$; but then $\gamma\in G\cap U$ so $\gamma = 1$ and $g\in G\cap G^\prime$. Therefore $\overline \phi$ induces an injective anti-morphism $G/(G\cap G^\prime) \rightarrow G^\prime/(G\cap G^\prime)$, which is bijective as $|G| = |G^\prime|$.
\end{proof}

\begin{proposition} \label{prop:simple}
Let $G$ be a finite simple group.
\renewcommand{\theenumi}{\alph{enumi}}
\renewcommand{\labelenumi}{(\alph{enumi})}
\begin{enumerate}
\item
If $G$ is a normal complement to a subgroup $U$ in a group $\Gamma$, then every normal complement of $U$ in $\Gamma$ is isomorphic to $G$. 
\item
Hence if $E/k$ is a pre-Galois extension for which the simple group $G$ is a pre-Galois group, then every 
pre-Galois group of $E/k$ is isomorphic to $G$.
\end{enumerate}
\end{proposition}

\begin{proof}
The second assertion is immediate from the first assertion and Corollary~\ref{prop:potentiallyGoG_characterization_cor}(\ref{pot cor complem}).  For the first assertion,
let $G'$ be another normal complement of $U$ in $\Gamma$.  
As $G\cap G^\prime$ is normal in $G$, it follows that $G\cap G^\prime=\{1\}$.  Thus there is an 
anti-isomorphism $\phi:G \to G'$ by Lemma~\ref{lem:unique-complement}; and hence the map $\phi^\ast:G\rightarrow G^\prime$ sending each $g\in G$ to $\phi(g)^{-1}$ is an isomorphism. 
\end{proof}

As the following example shows, it is possible for a field extension
to be pre-Galois with respect to one group, and to be potentially Galois but {\em not} pre-Galois with respect to a different group.  In particular, among the minimal field extensions $\LL(G)$ ($G\in \mathscr G$), 
it is possible that one of them is Galois over $k$ and another is not.

\begin{example} \label{preGal and potGal ex}
Let ${\gp} = S_4$ and let $M$ be the subgroup of permutations fixing $4$; thus 
$M$ is isomorphic to $S_3$.
The Klein four subgroup $G$ of $S_4$, generated by the transpositions $(1\ 2)$ and $(3\ 4)$,
 is a normal complement of $M$.  
But the cyclic group $G'$ generated by $(1\ 2\ 3\ 4)$ is a non-normal complement of $M$,  
and $G'$ is not isomorphic to $G$.
Now let $\widehat E/k$ be a Galois field extension of group ${\gp}$ and let $E$ be the fixed field of $M$, so that $[E:k]=4$ and $\widehat E$ is the Galois closure of $E/k$.  Let $L,L'$ be the fixed fields in $\widehat E$ of $G,G'$, respectively.  Then $E \otimes_k L = E \otimes_k L' = \widehat E$ is Galois over $L$ with group $G$, and is Galois over $L'$ with group $G'$. So 
$E/k$ is potentially Galois of group $G$, and of group $G'$. Furthermore $L/k$ is Galois (as $G$ is normal in $\Gamma$), so $E/k$ is pre-Galois of group $G$. On the other hand, $\Gamma$ has no normal cyclic subgroup of order $4$; in particular, $M$ has no normal cyclic complement (of order $4$). So, from Corollary \ref{prop:potentiallyGoG_characterization_cor}(\ref{pot cor SZ}$), E/k$ is not pre-Galois with respect to the cyclic group of order $4$. Summarizing, $E/k$ is pre-Galois with respect to the Klein four group, whereas it is potentially Galois but not pre-Galois with respect to the cyclic group of order four.
\end{example}

Given a field $k$ and two finite field extensions $E$ and $L$ of $k$, if $E$ is Galois over $k$ then the compositum $EL$ is Galois over $L$.  But the corresponding property does not in general hold for potentially Galois extensions, as the next example shows.

\begin{example}
Consider a Galois extension $N/k$ of group $S_6$. Let $H$ be the subgroup
generated by the $6$-cycle $(1\  2\  3\  4\  5\  6)$ and let
and $E=N^H$ be the fixed field of $H$ in $N$. The Galois closure of $E/k$ is $N/k$, and $H$ 
has a complement in $S_6$, e.g.\ the copy of $S_5\subset S_6$ fixing $6$. Hence 
$E/k$ is potentially Galois, of group $S_5$. 

Next, let $L=N^{A_6}$ be the fixed field of $A_6\subset S_6$ in $N$. Then
\begin{itemize}
\item $NL = N$ (as $L\subset N$),
\item $EL =  N^H N^{A_6} = N^{H\cap A_6}$,
\item $N/EL$ is Galois, and its Galois group is the order $3$ subgroup  $H\cap A_6$ generated by 
the square of $(1\  2\  3\  4\  5\  6)$. 
\end{itemize}
Thus $N/L$ is the Galois closure of $EL/L$, and $\Gal(N/EL)$ has no complement 
in the group $\Gal(N/L)=A_6$ (as such a complement would be of order $120$ and $A_6$
has no subgroup of order $120$). Therefore $EL/L$ is not potentially Galois, even though $E/k$ is potentially Galois.
\end{example}

\subsection{Potential Galois correspondence} \label{pot Gal corresp}

As we discuss next, there is an analog of the Galois correspondence for potentially Galois extensions (see Question~\ref{inv_ppGal}(\ref{pot Gal cor q})).

Let $E/k$ be a degree $d$ potentially Galois extension and $L/k$ be an extension with $EL/L$ Galois 
of degree $d$. Let $\widehat E$ be the Galois closure of $E$ over $k$. We produce below a 1-1 correspondence between the set of sub-extensions of $E/k$ and a certain subset of subgroups of the potential Galois group $\GG(L)={\rm res}(\Gal(EL/L))$, where ${\rm res}$ is as before the restriction map
$\Gal(EL/L) \rightarrow \Gal(\widehat E/k)$.

Set ${\gp}=\Gal(\widehat E/k)$ and ${\gp}_E=\Gal(\widehat E/E)$.  
Let $\mathrm{Subgp}_{\gp}(\GG(L))$ be the set of subgroups $H\subset \GG(L)$ such that the product set 
$H{\gp}_E := \{hg\,|\,h \in H, g \in {\gp}_E\}$ is a subgroup of 
${\gp}$ (or equivalently, such that $H{\gp}_E={\gp}_EH$).  Let $\mathrm{Subfld}_k(E)$ be the set of subfields of $E$ that contain $k$.  We then define maps between these sets as follows:

\begin{itemize}
\item
$\varepsilon_L:\mathrm{Subgp}_{\gp}(\GG(L)) \to \mathrm{Subfld}_k(E)$, $\varepsilon_L(H) = E^H := E \cap \widehat E^H$.
\item
$\eta_L:\mathrm{Subfld}_k(E) \to \mathrm{Subgp}_{\gp}(\GG(L))$, $\eta_L(E^\prime/k) = {\rm res}(\Gal(EL/E^\prime L))$.
Note that for every sub-extension $E^\prime/k$ of $E/k$, the extension $EL/E^\prime L$ is Galois of degree $[E:E^\prime]$ because $E/k$ is potentially Galois over $L$; so $E/E^\prime$ is potentially Galois over $L$.
\end{itemize}

\begin{proposition} \label{prop:Galois_correspondence} 
Let $E/k$ be a finite field extension that is potentially Galois over $L$.
Let $\widehat E$ be the Galois closure of $E$ over $k$, and set ${\gp}=\Gal(\widehat E/k)$ and ${\gp}_E=\Gal(\widehat E/E)$.
Then the maps $\varepsilon_L$ and $\eta_L$ as above define a bijective ``potential Galois correspondence'' between the set of subgroups  
$\mathrm{Subgp}_{\gp}(\GG(L))$ and the set of subfields $\mathrm{Subfld}_k(E)$.

More precisely, for every subfield $E' \in \mathrm{Subfld}_k(E)$ and every subgroup $H \in \mathrm{Subgp}_{\gp}(\GG(L))$,
\begin{itemize}
\item 
$\varepsilon_L(H) = E^H/k$  is a sub-extension of $E/k$ of subdegree equal to $|H|$ {\rm (}i.e. $[E:E^H] = |H|${\rm )},
\item
$\eta_L(E^\prime/k)= {\rm res}(\Gal(EL/E^\prime L))$ is a subgroup of $\GG(L)$ such that $\eta_L(E^\prime/k) \hskip 1pt {\gp}_E$ is a subgroup of ${\gp}$ and is of order $[E:E^\prime]$, 
\item
$\varepsilon_L\circ \eta_L(E^\prime/k) = E^\prime/k$ and $\eta_L \circ \varepsilon_L (H) = H$.
\end{itemize}
Furthermore, if $L/k$ is Galois and $H$ is a characteristic subgroup of $\GG(L)$, then $H {\gp}_E$ is a subgroup of ${\gp}$ and the extension $\varepsilon_L(H) = E^H/k$ is  pre-Galois over $L$, of group $\GG(L)/H$.
\end{proposition}

The proof of Proposition~\ref{prop:Galois_correspondence} starts with this pure group-theoretical lemma.

\begin{lemma} \label{lemma:alg_correspondence}
Let ${\gp} = GA$ be the ZS-product of two subgroups $G$ and $A$. There is an index preserving 1-1 correspondence $\varepsilon$ between the collection ${\mathcal F}_{G}(A)$ of subgroups $H\subset G$ such that $HA$ is a subgroup of ${\gp}$ and the collection ${\mathcal N}_{{\gp}}(A)$ of the subgroups of ${\gp}$ that contain $A$.
\vskip 1mm

\noindent
More precisely, this 1-1 correspondence is given by the maps 
$$\begin{matrix}
\varepsilon: &{\mathcal F}_{G}(A) & \rightarrow & {\mathcal N}_{{\gp}}(A) \cr
&H& \mapsto & HA \cr
\end{matrix}
\hskip 8mm \hbox{\it and} \hskip 8mm
\begin{matrix}
\varepsilon^{-1}: &{\mathcal N}_{{\gp}}(A)  & \rightarrow & {\mathcal F}_{G}(A)\cr
&H^\prime& \mapsto & G\cap H^\prime \cr
\end{matrix}
$$ 
\end{lemma}

\begin{proof}
For $H\in  {\mathcal F}_{G}(A)$ and $H^\prime \in {\mathcal N}_{{\gp}}(A)$, we have 
$$|{\gp}|/|\varepsilon(H)|=|{\gp}|/|H A| = |G| |A| /|H| |A| = |G|/|H|$$ 
\noindent
and
$$|G|/\varepsilon^{-1}(H^\prime)|=|G|/|G\cap H^\prime| = |G H^\prime|/|H^\prime| = |{\gp}|/|H^\prime|$$

\noindent
(as $G H^\prime \supset GA={\gp}$). Thus the maps $\varepsilon$ and $\varepsilon^{-1}$ preserve the indices. 

Now the containment $(G\cap H^\prime)\hskip 1pt A \subset H^\prime$ is clear; 
and using the above  we obtain
\[ [{\gp}:H^\prime] = [G:(G\cap H^\prime)] = [GA :(G\cap H^\prime)\hskip 1pt A] = [{\gp} :(G\cap H^\prime)\hskip 1pt A]. \]
Hence 
$(G\cap H^\prime) \hskip 1pt A = H^\prime$ and so $\varepsilon^{-1}(H') \in  {\mathcal F}_{G}(A)$. 

It is then easily checked that the map $\varepsilon:  {\mathcal F}_{G}(A)\rightarrow {\mathcal N}_{{\gp}}(A)$ and $\varepsilon^{-1}: {\mathcal N}_{{\gp}}(A) \rightarrow  {\mathcal F}_{G}(A)$ are inverse to each other. Namely $\varepsilon \circ \varepsilon^{-1}= {\rm Id}$ means that 
$(G\cap H^\prime)\hskip 1pt A = H^\prime$ for every $H^\prime\in {\mathcal N}_{{\gp}}(A)$, which has just been checked. And we have $\varepsilon^{-1} \circ \varepsilon = {\rm Id}$: if $H\in  {\mathcal F}_{G}(A)$, then $(HA) \cap G = H$; the containment $\supset$ is clear and the converse one easily follows from $G\cap A=\{1\}$.
\end{proof}

\begin{proof}[Proof of Proposition \ref{prop:Galois_correspondence}]
Let $A = {\gp}_E = \Gal(\widehat E/E)$.
In the notation of Lemma~\ref{lemma:alg_correspondence}, $\mathrm{Subgp}_{\gp}(\GG(L)) = {\mathcal F}_{\GG(L)}(A)$.
Note that the group ${\gp} = \Gal(\widehat E/k)$ is the ZS-product of $\GG(L)$ with $A$.  
So Lemma~\ref{lemma:alg_correspondence} applies, and we obtain a bijection 
$\varepsilon: \mathrm{Subgp}_{\gp}(\GG(L)) \to {\mathcal N}_{{\gp}}(A)$.
We will show that the  ``potential Galois correspondence'' of Proposition \ref{prop:Galois_correspondence} is obtained by composing this bijection with the classical Galois correspondence $\gamma:{\mathcal N}_{{\gp}}(A) \to \mathrm{Subfld}_k(E)$. 
\vskip 0,5mm

First, for $H\in \mathrm{Subgp}_{\gp}(\GG(L))$, we have
$\gamma \circ \varepsilon (H) = \widehat E^{HA}/k = E^H/k = \varepsilon_L (H)$. Second, to check that 
$\varepsilon^{-1} \circ \gamma^{-1} = \eta_L$, we introduce the unique $L_0\in \mathscr L$ such that 
$L_0\subset L$ (see Theorem~\ref{prop:potentiallyGoG_characterization}); recall that $\GG(L)=\GG(L_0)$.
For $E^\prime/k \in \mathrm{Subfld}_k(E)$, we have:
\begin{align*}
\varepsilon^{-1} \circ \gamma^{-1}(E^\prime/k) &= \varepsilon^{-1}(\Gal(\widehat E/E^\prime))\\
&=\GG(L)\cap \Gal(\widehat E/E^\prime)\\
&=\GG(L_0) \cap \Gal(\widehat E/E^\prime)\\
&=\Gal(\widehat E/L_0)\cap \Gal(\widehat E/E^\prime)\\
&=\Gal(\widehat E/E^\prime L_0)\\
&=\Gal(E L_0 /E^\prime L_0)\\
&={\rm res}(\Gal(E L /E^\prime L))  = \eta_L(E^\prime/k). 
\end{align*}

For the final part of of Proposition \ref{prop:Galois_correspondence}, assume that $L/k$ is Galois. Then the 
extension $E/k$ is  pre-Galois and $\widehat E/k$ is Galois of group a semi-direct product $\GG(L)\rtimes {\gp}_E$.
Being a characteristic subgroup of $\GG(L)$ implies for the subgroup $H$ that it is a normal subgroup of  $\GG(L)\rtimes {\gp}_E$. In particular, $H {\gp}_E$ is a subgroup of ${\gp}$. By the above, $\varepsilon_L(H) = E^H/k$ is potentially Galois over $L$, and so pre-Galois over $L$ (as $L/k$ is Galois), and $E^HL/L$ is Galois of group $\GG(L)/H$.
\end{proof}

\begin{corollary} 
Let $E/k$ be a finite field extension that is potentially Galois over $L$. As above,
set ${\gp}=\Gal(\widehat E/k)$ and ${\gp}_E=\Gal(\widehat E/E)$.
\renewcommand{\theenumi}{\alph{enumi}}
\renewcommand{\labelenumi}{(\alph{enumi})}
\begin{enumerate}
\item \label{Galois corres-corol-a}
The base change correspondence $E^\prime/k \mapsto E^\prime L/L$
yields a 1-1 correspondence between the set $\mathrm{Subfld}_k(E)$
and the set of sub-extensions $F/L$ of $EL/L$ such that the product set
$\Gal(EL/F)\, {\gp}_E$ is a subgroup of ${\gp}$.
\item \label{Galois corres-corol-b}
If $L,L^\prime \in \mathscr L$, the map $\eta_{L^\prime} \circ \varepsilon_L$ yields a 1-1 correspondence between the two sets $\mathrm{Subgp}_{\gp}(\GG(L))$ and $\mathrm{Subgp}_{\gp}(\GG(L^\prime))$. \end{enumerate}
\end{corollary}

\begin{proof}
By the usual Galois correspondence the latter set in (\ref {Galois corres-corol-a}) is in bijection with the set $\mathrm{Subgp}_{\gp}(\GG(L))$.  So (\ref {Galois corres-corol-a}) follows from the 1-1 correspondence between $\mathrm{Subfld}_k(E)$ and $\mathrm{Subgp}_{\gp}(\GG(L))$ proved in Proposition~\ref{prop:Galois_correspondence}. Part (\ref{Galois corres-corol-b}) is clear since $\eta_{L^\prime}$ and $\varepsilon_L$ are bijective, by  Proposition~\ref{prop:Galois_correspondence}. 
(Part~(\ref{Galois corres-corol-b}) can also be deduced from a pure group-theoretical observation: with the notation of Lemma \ref{lemma:alg_correspondence},
if ${\gp} = GA = G^\prime A$ is the ZS-product of $G$ and $A$, and also of $G^\prime$ and $A$, then the correspondence $H\mapsto HA \cap G^\prime$ yields a 1-1 correspondence 
${\mathcal F}_{G}(A) \rightarrow {\mathcal F}_{G^\prime}(A)$. We leave the details to the reader.)
\end{proof}

\subsection{Related conditions} \label{subsec: Hopf}
A notion of ``Hopf Galois theory'' was introduced in \cite{ChaseSweedler} and studied further in \cite{greither-pareigis}.  That latter paper also considered conditions on field extensions that were more general than being Galois.
Given a finite separable field extension $E/k$ and a finite $k$-Hopf algebra $H$, they defined a notion of $E/k$ being ``$H$-Galois'' (pages~239-240 in \cite{greither-pareigis}), 
and a more restrictive notion of $E/k$ being ``almost classically Galois'' (pages~252-253 in \cite{greither-pareigis}). 

These two notions can each be formulated in terms of the left multiplication action of $\Gal(\widehat E/k)$ on the group ${\rm Perm}({\mathcal S})$ of permutations of the set $\mathcal S$ of left cosets of $\Gal(\widehat E/k)$ modulo $\Gal(\widehat E/E)$.  (As before, $\widehat E$ denotes the Galois closure of $E$ over $k$.)  Namely, according to \cite[Theorem 2.1]{greither-pareigis}, $E/k$ is $H$-Galois for some $k$-Hopf algebra $H$ if and only if there is a subgroup $G\subset {\rm Perm}({\mathcal S})$ acting freely and transitively
on ${\mathcal S}$ and which is normalized by the subgroup 
$\Gal(\widehat E/k)\subset {\rm Perm}({\mathcal S})$ (where the containment is via $\mu$).  By \cite[Proposition~4.1]{greither-pareigis}, $E/k$ is almost classically Galois
if and only if the above condition holds for some $G \subset \Gal(\widehat E/k)$.  The latter condition is strictly stronger than the former, by \cite[Corollary 4.4]{greither-pareigis}.  For these notions, partial analogs of the usual Galois correspondence were shown (with versions in \cite[Theorem~7.6]{ChaseSweedler}, \cite[Section~5]{greither-pareigis}, and \cite[Section~2]{CreRioVela}), though they do not include an analog of the usual bijection between normal subgroups and intermediate Galois extensions.

Meanwhile, from our Corollary \ref{prop:potentiallyGoG_characterization_cor}(\ref{pot cor SZ}), $E/k$ is 
potentially Galois if and only if there is a subgroup $G\subset \Gal(\widehat E/k)$ that acts freely and transitively on ${\mathcal S}$ via $\mu$, since this is equivalent to $G$ being a complement to $\Gal(\widehat E/E)$ in $\Gal(\widehat E/k)$.  Thus the condition of being almost classically Galois is a strengthening both of being $H$-Galois for some $H$ and also of being potentially Galois.
Moreover, given a finite separable extension $E/k$, with Galois closure $\widehat E/k$,
one of the equivalent conditions for $E/k$ to be almost classically Galois is that 
$\Gal(\widehat E/E)$ 
has a normal complement $G$ in $\Gal(\widehat E/k)$ (see \cite[Proposition~4.1]{greither-pareigis}).  So by Corollary~\ref{prop:potentiallyGoG_characterization_cor}(\ref{pot cor SZ}) above,
$E/k$ is almost classically Galois if and only if it is pre-Galois with group $G$.

By \cite[Introduction, Remark~3]{greither-pareigis}, it follows that $E/k$ is $H$-Galois if and only $\Spec(E)$ is a $\mu$-torsor over $k$, where the finite group scheme $\mu$ is $\Spec(H^*)$, and where $H^*$ is the dual Hopf algebra to $H$.  In particular, if $E/k$ is a Galois field extension of group $G$, 
then $\Spec(E)$ is a $G$-torsor, and $E/k$ is $H$-Galois with $H$ being the group ring $kG$; moreover, $E/k$ is almost classically Galois.  Thus every pre-Galois extension $K/k$ with group $G$ has the property that $\Spec(K)$ is a $\mu$-torsor over $k$ for some twisted form $\mu$ of $G$; but not conversely.
Also, every pre-Galois extension is potentially Galois, but not conversely (see Proposition~\ref{hilb flds pot Gal cond}(\ref{potentially but no pre})).  Thus

\centerline{$E/k$  pre-Galois $\Rightarrow$ $E/k$ potentially Galois + $\Spec(E)$ a torsor,}
\noindent 
and one can ask whether the converse holds:

\begin{question} \label{pre pot torsor}
Let $E/k$ be a finite separable field extension and let $G$ be a finite group.  If $E/k$ is potentially Galois with group $G$, and if $\Spec(E)$ is a $\mu$-torsor over $k$ where $\mu$ is a twisted form of $G$, must $E/k$ be pre-Galois with group $G$?
\end{question}

\begin{example}
Concerning the two hypotheses in Question~\ref{pre pot torsor} (being potentially Galois, and the spectrum being a torsor), we show by example that neither
implies the other.  
Note that otherwise, Question~\ref{pre pot torsor} could not have an affirmative answer, 
given the above comments about the separate converses not holding.
\renewcommand{\theenumi}{\alph{enumi}}
\renewcommand{\labelenumi}{(\alph{enumi})}
\begin{enumerate}
\item
Let $k=\Qq$ and let $E = k(\zeta_8) = k(\sqrt[4]{-1})$.  Then $\Spec(E)$ is $\mu_4$-torsor, and $\mu_4$ is a twisted form of $\Zz/4\Zz$.  But 
$E/k$ is a Galois extension with group $(\Zz/2\Zz)^2$; and so 
for any field extension $L/k$ such that $LE=L \otimes_k E$ is a field, $LE/L$ is also 
Galois with group $(\Zz/2\Zz)^2$.  Hence $LE/L$ is not Galois with group $\Zz/4\Zz$, and thus $E/k$ is not potentially Galois with group $\Zz/4\Zz$, even though $\Spec(E)$ is $\mu_4$-torsor.
\item
Let $E/k$ be a potentially Galois extension of degree $5$ that is not pre-Galois, as in 
Proposition~\ref{hilb flds pot Gal cond}(\ref{potentially but no pre}).  Thus its group is $\Zz/5\Zz$.  If $\Spec(E)$ is a $\mu$-torsor for some form $\mu$ of $\Zz/5\Zz$, then $\mu$ is given by an action of the absolute Galois group $\Gal(k)$ of $k$ on $\Zz/5\Zz$; i.e., by a homomorphism $\Gal(k) \to \Aut(\Zz/5\Zz) \cong \Zz/4\Zz$.  
The kernel of that action has index $1,2$, or $4$; and so $\mu$ becomes isomorphic to $\Zz/5\Zz$ over a Galois extension $k'/k$ of degree $1,2$, or $4$.  Since $k'/k$ is Galois of degree $5$, it follows that $k'$ and $E$ are linearly disjoint over $k$.  
So their compositum $k'E/k'$ has degree 5, and is isomorphic to $k'\otimes_k E$.  Thus $\Spec(k'E)$ is a $\mu_{k'}$-torsor over $k'$.  But $\mu_{k'}$ is just the constant group $\Zz/5\Zz$, and so $k'E/k'$ is a $\Zz/5\Zz$-Galois field extension.  
That says that $E/k$ is pre-Galois, and that is a contradiction.  Hence $\Spec(E)$ is not a $\mu$-torsor over $k$ for any form $\mu$ of $\Zz/5\Zz$, even though $E/k$ is potentially Galois with group $\Zz/5\Zz$.
(If $\Gal(\widehat E/k) = S_5$ or $A_5$, then one can also see that $\Spec(E)$ is not a $\mu$-torsor over $k$, or equivalently that $E/k$ is not $H$-Galois, by \cite[Corollary 4.8]{greither-pareigis}.)
\end{enumerate}
\end{example}


\section{Function field extensions} \label{sec:function-fields}

We now investigate the pre-Galois notion  in the 
context of finite extensions of function fields.
Fix an arbitrary field $k$
and a regular projective geometrically irreducible $k$-variety $B$ of positive dimension, e.g.\ $B=\Pp^1_\Qq$.
We indicate the separable closure of $k$ by $k^{\sep}$ and its absolute Galois group by $\Gal(k)$.

\subsection{Geometrically Galois extensions} \label{ssec:geometrically_Galois}
Let $F/k(B)$ be a finite separable extension. It is called {\it $k$-regular} if $F \cap \overline k = k$;
and in this case, $F$ is the function field of a geometrically irreducible branched $k$-cover $X \to B$.

We sometimes view finite $k$-regular extensions $F/k(B)$ as fundamental group representations $\phi: \pi_1(B \smallsetminus D, t)_k \rightarrow S_d$. Here $D$ denotes the branch divisor of the extension $F\overline k/\overline k(B)$,
$t\in B \smallsetminus D$ a fixed {\it base point}, $\pi_1(B \smallsetminus D, t)_k$ the $k$-fundamental group of $B \smallsetminus D$ and $d=[F:k(B)]$. We refer to \cite{DeDo1} or \cite{DeLe2} for more on the correspondence.

Let $\widehat F/k(B)$ be the Galois closure of $F/k(B)$. It is easily checked that, for any overfield $k^\prime$ of $k$,
the Galois closure of the extension $Fk^\prime/k^\prime(B)$ is the extension $\widehat F k^\prime/k^\prime(B)$ (where the field compositum of $\widehat F$ and $k^\prime$ is taken in an algebraic closure of $k^\prime(B)$). 

The Galois closure $\widehat F/k(B)$ is not $k$-regular in general. Its constant extension $\widehat F\cap \overline k$ is called the {\it constant extension in the Galois closure} of $F/k(B)$ and denoted by $\widehat k_F$. By definition, $\widehat F/\widehat k_F(B)$ is $\widehat k_F$-regular; hence $[\widehat F:\widehat k_F(B)] = [\widehat F k^\prime: k^{\prime}(B)]$ for every overfield $k^\prime$ of $\widehat k_F$.
Consequently the Galois groups $\Gal(\widehat F k^{\prime}/k^{\prime}(B))$ with $k^\prime \supset  \widehat k_F$ are all equal to the same group, called the {\it monodromy group} of the extension $F/k(B)$ and denoted by $G$. 

Note that 
$\widehat k_F \subset k^{\sep}$, hence $\widehat k_F  = \widehat F\cap k^{\sep}$. Furthermore the extension $\widehat k_F/k$ is Galois and $\Gal(\widehat k_F/k) \subset {\rm Nor}_{S_d}(G)/G$, where $d=[F:k(B)]$ and $G$ is viewed as a subgroup of $S_d$ via the {\it monodromy action} $G\hookrightarrow S_d$ of $G$ on the $k(B)$-embeddings $F \hookrightarrow\widehat F$ \cite[Proposition~2.3]{DeDo1}.
\vskip 1mm

We say that a finite $k$-regular extension $F/k(B)$ is {\it geometrically Galois}  if 
the field extension $F\overline k/\overline k(B)$ is Galois. This generalizes the definition
given in the introduction for $B=\Pp^1_k$. Also note that the condition is equivalent to 
$Fk^{\sep}/k^{\sep}(B)$ being Galois. Indeed, if $F\overline k/\overline k(B)$ is Galois, 
then $\widehat F \overline k= F\overline k$, and it follows from the equalities
\[[\widehat F k^{\sep}:k^{\sep}(B)]= [\widehat F \overline k:\overline k(B)] =  [F \overline k:\overline k(B)] = [F k^{\sep}:k^{\sep}(B)]\]
that $\widehat F k^{\sep} = F k^{\sep}$, and so that $F k^{\sep}/k^{\sep}(B)$ is Galois. The converse  is clear; i.e., if $F k^{\sep}/k^{\sep}(B)$ is Galois then $F\overline k/\overline k(B)$ is Galois.
\vskip 1mm

If $F/k(B)$ is geometrically Galois, then 
the Galois group $\Gal(F\overline k/\overline k(B))$ is the monodromy group $G$ of $F/k(B)$, of order $d=[F:k(B)]$. For every field $k^\prime \supset \widehat k_F$, since $[Fk^\prime:k^\prime(B)] = d$ and $[\widehat{F}k^\prime:k^\prime(B)] =  |G|$, it follows that $\widehat{F}k^\prime= Fk^\prime$ and hence $Fk^\prime/k^\prime(B)$ is Galois.  Note that if $F/k(B)$ is not itself Galois, then the field extension $\widehat k_F/k$ must therefore be non-trivial, and $\widehat F/k(B)$ then cannot be $k$-regular. (See also Remark~\ref{rem:constant-extension}(\ref{constant_extension}).)
Furthermore, $\Gal(\widehat k_F/k) \subset {\rm Aut}(G)$; viz., the monodromy action $G\hookrightarrow S_d$ is the regular representation, and ${\rm Nor}_{S_d}(G)/G$ is isomorphic to ${\rm Aut}(G)$.  (See the proof of \cite[Proposition 3.1]{DeDo1}; see also \cite[Theorem~3.2]{Has15}.)

\begin{proposition} \label{prop:geom vs pot} 
Let $k$ be a field.
\renewcommand{\theenumi}{\alph{enumi}}
\renewcommand{\labelenumi}{(\alph{enumi})}
\begin{enumerate}
\item \label{geom Gal pre}
Every geometrically Galois extension $F/k(T)$ is  pre-Galois: more precisely, for every Galois extension $k^\prime/k$ with $k^\prime \supset \widehat k_F$, the extension $F/k(T)$ is pre-Galois
over $k^{\prime}(T)$, and the corresponding pre-Galois group is the monodromy group $\Gal(F\overline k/\overline k(T))$.
\item \label{pre not geom}
The converse of the first assertion in~(\ref{geom Gal pre}) does not hold: there are $k$-regular pre-Galois extensions of $k(T)$ that are not geometrically Galois.
\item \label{some pre are geom}
The following partial converse of~(\ref{geom Gal pre}) holds: 
Consider a $k$-regular field extension $F/k(T)$ such that there exists a Galois extension $L/k(T)$ such that $FL/L$ is Galois and the extensions $F\overline k/\overline k(T)$ and $L\overline k/\overline k(T)$ have no common branch point (and so $F\overline k/\overline k(T)$ and $L\overline k/\overline k(T)$ are linearly disjoint).  Then
the extension $F/k(T)$ is geometrically Galois. 
\end{enumerate}
\end{proposition}

Part~(\ref{pre not geom}) of this result shows that the answer to Question~\ref{always_pp}(\ref{preGal geomGal}) is in general ``no''.

\begin{proof}
Part~(\ref{geom Gal pre}) 
follows from the above observation that if $k^\prime \supset \widehat k_F$ then $F k^\prime/k^\prime(T)$ is Galois.

For part~(\ref{pre not geom}), take a Galois extension $N/\Cc(T)$ of group ${\mathcal G} = G\rtimes H$ with $H$ not normal in ${\mathcal G}$ acting on $G$ (such an extension exists as the Inverse Galois Problem is solved over $\Cc(T)$). For $F=N^H$, $F/\Cc(T)$ is not Galois and not geometrically Galois 
either ($\Cc$ is algebraically closed) but is  pre-Galois (Corollary \ref{prop:potentiallyGoG_characterization_cor}(\ref{pot cor SZ})).

For part~(\ref{some pre are geom}), by Remark \ref{rem:post-thm}(\ref{corresp rk b}), 
one may take $L/k(T)$ so that $FL=\widehat F$. But then each branch point of $L\overline k/\overline k(T)$ is a branch point of $\widehat F\overline k /\overline k(T)$ and so is a branch point of $F\overline k/\overline k(T)$. Therefore $L\overline k/\overline k(T)$ is unramified eve\-ry\-where, which gives $L\overline k = \overline k(T)$
hence $L\subset \overline k(T)$. Finally we have $[F\overline k: \overline k(T)] = [F:k(T)]$ since $F/k(T)$ is $k$-regular, and $F\overline k/\overline k(T)$ is Galois since $FL/L$ is already Galois, so $F/k(T)$ is indeed geometrically Galois.
\end{proof}

Although a geometrically Galois extension of $k(T)$ has a unique geometric Galois group (by definition), it can have more than one group when viewed as a pre-Galois extension of $k(T)$, as the following geometric analog of Example \ref{ex:two-normal-complements} shows.

\begin{example} \label{geom also pre other gp}
Let $k=\Qq$ and $F=\Qq(T^{1/4})$.  The extension $F/\Qq(T)$ is geometrically Galois with group $\Zz/4\Zz$, since $F/\Qq(T)$ becomes Galois with group $\Zz/4\Zz$ over $\Qq(i)(T)$; and hence it is also pre-Galois with group $\Zz/4\Zz$.  The Galois group $\Gamma$ of its Galois closure $\widehat F/\Qq(T)$ is isomorphic to $D_8$, the dihedral group of order $8$.  The $4$-cyclic subgroup $G=\Gal(F(i)/\Qq(i)(T)) \subset \Gamma$ is a normal complement to the $2$-cyclic subgroup $\Gal(\widehat F/F)$; but the latter group also has a normal complement $G'\cong(\Zz/2\Zz)^2$ in $\Gamma$
(see Example~\ref{ex:two-normal-complements}).  
Thus $F/k(T)$ is also pre-Galois over $L=\widehat F^{G'}$, with group $(\Zz/2\Zz)^2$.   Note that $L$ cannot be of the form $k'(T)$: otherwise 
$G' = \Gal(Fk'/k'(T))$ would 
be $G$, as it contains $\Gal(F\overline k/\overline k(T))=G$ and has the same order.  (This can also be seen explicitly.)  
\end{example}

\begin{remark}
The geometric analog of the situation considered in Section~\ref{subsec: Hopf} is simpler than the one there:  An extension of $k(T)$ is geometrically Galois if and only if it is a torsor and it is regular.  For the former implication, every geometrically Galois extension is pre-Galois (by Proposition~\ref{prop:geom vs pot}(\ref{geom Gal pre})) and hence a torsor (by the comments in Section~\ref{subsec: Hopf}); and it was noted in the Introduction that it is regular.  
Conversely, suppose that an extension of $F/k(T)$ is a $\mu$-torsor for some finite group scheme $\mu$, and that it is regular.  Then $F$ is linearly disjoint from $\overline k$ over $k(T)$, and so $F \otimes_k \overline k$ is a field and a $\mu_{\overline k}$-torsor over $\overline k(T)$.  But since $\overline k$ is algebraically closed, $\mu_{\overline k}$ is a constant finite group $G$.  Thus $F \otimes_k \overline k$ is a Galois field extension of $\overline k(T)$ having group $G$; and so $F/k(T)$ is geometrically Galois.  
\end{remark}

\subsection{Specialization} \label{ssec:specialization} 
In this subsection, we extend classical results about specializing Galois function field extension to our pre-Galois context.  In particular, we prove Proposition \ref{intro:IGP-implications} over Hilbertian fields.
As in Section~\ref{ssec:geometrically_Galois}, $k$ is an arbitrary field and $B$ is a regular projective geometrically irreducible 
$k$-variety $B$ of positive dimension. 
 
Given a finite $k$-regular extension $F/k(B)$ and a point $t_0\in B(k)$ not in the branch divisor $D$, the specialization $F_{t_0}/k$ of $F/k(B)$ at $t_0$ is defined as follows: if ${\rm Spec}(V)$ is an affine neighborhood (for the Zariski topology) of $t_0$, which corresponds to some maximal ideal ${\frak p}_{t_0}$ such that $V/{\frak p}_{t_0}=k$ and $U$ is the integral closure of $V$ in $F$, then $F_{t_0}$ is the $k$-\'etale algebra $U\otimes_VV/{\frak p}_{t_0}$ (it does not depend on the affine subset ${\rm Spec}(V)$ and is defined up to $k(B)$-isomorphism). 
If $F/k(B)$ corresponds to the fundamental group representation  $\phi:\pi_1(B\smallsetminus D,t)_k \rightarrow S_d$ and $s_{t_0}: \Gabs(k) \rightarrow  \pi_1(B\smallsetminus D,t)_k$ is the section associated to $t_0$, $F_{t_0}/k$ is the \'etale algebra associated with the map $\phi \circ s_{t_0}:\Gabs(k) \rightarrow S_d$. The $k$-algebra $F_{t_0}$ is a field if and only if $\phi \circ s_{t_0}(\Gabs(k))$ is a transitive subgroup of $S_d$. 

\begin{proposition}  \label{prop:specialization} 
Let $B$ be a regular projective geometrically irreducible variety of positive dimension over a field $k$.
Let $F/k(B)$ be a separable degree $d$ extension and $\widehat F/k(B)$ be its Galois closure.
\begin{enumerate}
\renewcommand{\theenumi}{\alph{enumi}}
\renewcommand{\labelenumi}{(\alph{enumi})}
\item \label{pot Gal specn}
Assume $F/k(B)$ is potentially Galois of group $G$. For every \hbox{point $t_0$} in $B(k)\smallsetminus D$  such that 
$[(\widehat F)_{t_0} : k] = [\widehat F:k(B)]$,  the extension $F_{t_0}/k$ is a degree $d$ field extension 
that is potentially Galois of group $G$. Furthermore $F_{t_0}/k$ is  pre-Galois if $F/k(T)$ is.
\item \label{geom Gal specn}
 Assume $F/k(B)$ is geometrically Galois of group $G$ and let $\widehat k_F/k$ be the constant extension in the Galois closure $\widehat F/k(B)$ of $F/k(B)$. For every point $t_0\in B(k)\smallsetminus D$  such that  $F_{t_0}$ is a field and $[F_{t_0}\widehat k_F : \widehat k_F] = d$, we have the following, where $\widehat{F_{t_0}}/k$ is the Galois closure  of $F_{t_0}/k$:
\begin{enumerate}
[label=(\roman*)]
\renewcommand{\theenumi}{\roman{enumi}}
\item \label{specn Gal cl}
$\widehat{F_{t_0}} = F_{t_0} \widehat k_F = (\widehat F)_{t_0}$.
\item \label{specn pre Gal}
the extension $F_{t_0}/k$ is  pre-Galois over $\widehat k_F$
of group $G$. 
\end{enumerate}
\end{enumerate}
\end{proposition}

\begin{proof} (\ref{pot Gal specn}): By Theorem \ref{prop:potentiallyGoG_characterization}, there is a sub-extension $L/k(B)$ of $\widehat F/k(B)$ such that $\widehat F= FL$ and $FL/L$ is Galois with group $G$; and the Galois group $\Gal(\widehat F/k(B))$ is the ZS-product of $G$ and $\Gal(\widehat F/F)$. 

Let $t_0\in B\smallsetminus D$ such that $[(\widehat F)_{t_0} : k] = [\widehat F:k(B)]$.
Then 
\begin{itemize}
\item $(\widehat F)_{t_0}/k$ is a Galois field extension of group $\Gal(\widehat F/k(B))$,
\item $F_{t_0}/k$ is an extension of degree $d$,
\item $(\widehat F)_{t_0}/F_{t_0}$ is a Galois field extension of group $\Gal(\widehat F/F)$.
\end{itemize}
It follows that $(\widehat F)_{t_0}/k$ is a Galois field extension containing $F_{t_0}/k$ and its Galois group
is the ZS-product of $G$ and $\Gal((\widehat F)_{t_0}/F_{t_0})$.   By 
Proposition \ref{ZS_equiv}, 
$F_{t_0}/k$ is potentially Galois over the fixed field $((\widehat F)_{t_0})^G$. 
If $F/k(B)$ is further assumed to be  pre-Galois, the original extension $L/k(B)$ may be assumed to be Galois and then $G$ is normal in $\Gal(\widehat F/k(B))$. Hence the extension $((\widehat F)_{t_0})^G/k$ is Galois as well and $F_{t_0}/k$ is  pre-Galois of group $G$.

\smallskip

(\ref{geom Gal specn}): It follows from $F\widehat k_F= \widehat F$ that $F_{t_0} \widehat k_F \subset (\widehat F)_{t_0}$. As $[F_{t_0}\widehat k_F : \widehat k_F] = d$ and $\Gal((\widehat F)_{t_0}/\widehat k_F) \subset  \Gal(\widehat F/\widehat k_F(T))=G$, we have $F_{t_0} \widehat k_F = (\widehat F)_{t_0}$ and $F_{t_0}\widehat k_F/\widehat k_F$ is Galois of group $G$, {\it i.e.} assertion (\ref{geom Gal specn})\ref{specn pre Gal} holds. It remains to prove $\widehat{F_{t_0}} = (\widehat F)_{t_0}$, which we do below. 

\smallskip

It follows from $F_{t_0} \subset (\widehat F)_{t_0}$ that $\widehat{F_{t_0}} \subset (\widehat F)_{t_0}$. For the other containment $(\widehat F)_{t_0}\subset \widehat{F_{t_0}}$, since $(\widehat F)_{t_0}=F_{t_0}\widehat k_F$, we have to prove that $\widehat k_F \subset \widehat{F_{t_0}}$.
Denote by $\phi:\pi_1(B\smallsetminus D,t)_k \rightarrow S_d$ the fundamental group representation of $F/k(T)$ and by $s_{t_0}: \Gabs(k) \rightarrow  \pi_1(B\smallsetminus D,t)_k$ the section associated to $t_0$. 
Proving $\widehat k_F \subset \widehat{F_{t_0}}$ amounts to showing that ${\rm ker}(\phi \circ s_{t_0}) \subset \Gabs({\widehat k_F})$. Let $\tau \in {\rm ker}(\phi \circ s_{t_0})$, {\it i.e.} $s_{t_0}(\tau)\in {\rm ker}(\phi)$. Thus $s_{t_0}(\tau)$
fixes $\widehat F =F\widehat k_F$ and in particular fixes $\widehat k_F$. Hence $\tau \in \Gabs({\widehat k_F})$, which finishes the proof. \end{proof}

Along the lines of the implication ({Regular} IGP/$k$) $\Rightarrow$ (IGP/$k$), we have the following result for any Hilbertian field (e.g., a number field), via specialization:

\begin{proposition} \label{intro:IGP-implications}
Let $k$ be a Hilbertian field and let $G$ be a finite group.  If $G$ is a pre-Galois group over $k(T)$, then $G$ is a pre-Galois group over $k$. 
\end{proposition}

\begin{proof}
By hypothesis, there exists a pre-Galois extension $F/k(T)$ of group $G$. For $B=\Pp^1$, the set of $t_0\in k$ such that $[(\widehat F)_{t_0} : k] = [\widehat F:k(B)]$ contains a separable Hilbert subset of $k$, which is infinite since $k$ is Hilbertian. Using Proposition~\ref{prop:specialization}(\ref{pot Gal specn}), we obtain that $G$ is a pre-Galois group over $k$, namely the pre-Galois group 
of some specialization $F_{t_0}/k$ of $F/k(T)$.
\end{proof}

Combining Proposition~\ref{intro:IGP-implications} with Proposition~\ref{prop:geom vs pot}(\ref{geom Gal pre}), we have implications
\vskip 1mm
\centerline{\rm (Geometric IGP/$k$) $\Rightarrow$ (Pre-IGP/$k(T)$) $\Rightarrow$ 
(Pre-IGP/$k$)}
\vskip 1mm

\noindent for $k$ Hilbertian.
These problems offer a natural graduation of inverse Galois theory. 
Note that the implication (Geometric IGP/$k$) $\Rightarrow$ 
(Pre-IGP/$k$) can also be seen directly in this situation by using Proposition~\ref{prop:specialization}(\ref{geom Gal specn}).  Namely, consider the set of $t_0\in k$ such that  $F_{t_0}$ is a field and $[F_{t_0}\widehat k_F : \widehat k_F] = d$.  This set contains a separable Hilbert subset of $k$, which is necessarily infinite if $k$ is Hilbertian.
\vskip 2mm

\subsection{Lifting Problems} \label{ssec:lifting_problems}
 We turn to Question~\ref{geomBB}, which is a weakening of the arithmetic lifting problem (or Beckmann-Black problem), in which we ask for a geometrically Galois extension, rather than a regular Galois extension, that lifts a given Galois extension of the field. Theorem~\ref{geom BB thm} and Corollary~\ref{BB ample} provide some answers. 

Recall that  $B$ is a regular projective geometrically irreducible $k$-variety.  

The definition of the twisted extension $\widetilde F^E/k(B)$ appearing in the statement below is recalled in the proof, together with 
its main properties. For more on twisting, we refer to \cite{DeLe2} or in \cite{DeBB} (where a form of Proposition \ref{prop:constant-extension-prescribed}(\ref{specn et alg}) already appears).

\begin{proposition} \label{prop:constant-extension-prescribed}
Let  $G$ be a group with trivial center and $F/k(B)$ be a $k$-regular Galois \hbox{extension of group} $G$ 
such that the corresponding cover of $B$ has a $k$-rational point above some point $t_0\in B(k)$ not in the branch divisor.
Let $E/k$ be a Galois extension of group $H$ isomorphic to a subgroup of $G$. Consider the  extension $\widetilde F^E/k(B)$ obtained by twisting $F/k(B)$ by $E/k$.
\begin{enumerate}
\renewcommand{\theenumi}{\alph{enumi}}
\renewcommand{\labelenumi}{(\alph{enumi})}
\item \label{const ext geom Gal}
Then the extension $\widetilde F^E/k(B)$  is geometrically Galois of group $G$.  Its Galois closure is the extension $FE/k(B)$, which has
Galois group  $G\times H$.
\item \label{const ext in Gal closure}
The constant extension in the Galois closure of $\widetilde F^E/k(B)$ is $E/k$. 
\item \label{specn et alg}
The specialization of $\widetilde F^E/k(B)$ at $t_0$ is the $k$-\'etale algebra corresponding to $(G:H)$ copies of $E/k$.
\end{enumerate}
\end{proposition}

\begin{proof} 
One may assume $H\subset G$. Let  $\phi:\pi_1(B\smallsetminus D,t)_k \rightarrow G$ be  the fundamental group representation of $F/k(B)$ and $\varphi:\Gabs(k) \rightarrow G$ be a Galois representation of $E/k$. 
Write ${\rm Perm}(G)$ for the group of permutations of the set $G$. The twisted extension $\widetilde F^E/k(B)$ corresponds to the fundamental group representation $\widetilde \phi^E:\pi_1(B\smallsetminus D,t)_k \rightarrow {\rm Perm}(G)$ given as follows: for each element of $\pi_1(B\smallsetminus D,t)_k$ uniquely written $x \hskip 1pt s_{t_0}(\tau)$ (with $x\in \pi_1(B\smallsetminus D,t)_{k^{\hbox{\sevenrm sep}}}$, $\tau \in \Gabs(k)$ and $s_{t_0}: \Gabs(k) \rightarrow  \pi_1(B\smallsetminus D,t)_k$ the section associated to $t_0$), we have, for every $g\in G$,
\[\widetilde \phi^E(x \hskip 1pt s_{t_0}(\tau)) (g) = \phi(x \, s_{t_0}(\tau)) \, g \ \varphi(\tau)^{-1}= \phi(x)  \, g \, \varphi(\tau)^{-1}\]
(we have $\phi( s_{t_0}(\tau)) = 1$ as there is a $k$-rational point above $t_0$).

The extension $\widetilde F^E/k(B)$ becomes isomorphic to $F/k(B)$ after scalar extension to $k^\sep$ (as $\phi$ and $\widetilde \phi^E$ have the same restriction on $\pi_1({\Pp}^1\smallsetminus{\bf t})_{k^{\hbox{\sevenrm sep}}}$). The same is true {\it a fortiori} over $\overline k$, hence $\widetilde F^E/k(B)$ is geometrically Galois of group $G$. Furthermore, using that $Z(G)=\{1\}$, one easily checks that 
\[{\rm ker}(\widetilde \phi^E) = {\rm ker}(\phi) \cap {\rm ker}(\varphi),\]
which indeed shows that the Galois closure of $\widetilde F^E/k(B)$ is the extension $FE/k(T)$.
This proves (\ref{const ext geom Gal}).

The constant extension in the Galois closure of $\widetilde F^E/k(B)$ is given by the map $\Gabs(k) \rightarrow {\rm Nor}_{{\rm Perm}(G)}(G)/G$ induced on $\Gabs(k)$ by $\widetilde \phi^E$ (this is detailed in \cite[Section 2.8]{DeDo1}).
Here $G\hookrightarrow {\rm Perm}(G)$ is the left-regular representation of $G$. Then ${\rm Nor}_{{\rm Perm}(G)}(G)/G$
is  the image of $G$ via the right-regular representation of $G$. Taking into account that $Z(G)=\{1\}$, the map $\Gabs(k) \rightarrow {\rm Nor}_{{\rm Perm}(G)}(G)/G$ can be identified to the one sending each $\tau \in \Gabs(k)$ to the right multiplication map $g\mapsto g \cdot \varphi(\tau)^{-1}$. Its kernel is the same as the initial representation $\varphi: \Gabs(k) \rightarrow G$. This proves (\ref{const ext in Gal closure}).

Statement (\ref{specn et alg}) follows as well as the specialization of $\widetilde F^E/k(B)$ at $t_0$ corresponds to the map $\widetilde \phi^E\circ s_{t_0}: \Gabs(k)\rightarrow {\rm Nor}_{{\rm Perm}(G)}(G)$ which also sends each $\tau \in \Gabs(k)$ to the right multiplication $g\mapsto g \cdot \varphi(\tau)^{-1}$. The stabilizer of a given $g\in G$ is ${\rm ker}(\varphi)$ and the corresponding fixed field is~$E$.
\end{proof}

In particular, we obtain an affirmative answer to Question~\ref{geomBB} in the above situation:

\begin{theorem} \label{geom BB thm}
Let $G$ be a finite group with trivial center, and let $k$ be a field.  
Suppose there is a $k$-regular Galois extension $F/k(T)$ with Galois group $G$, whose corresponding cover $\pi:X \to\Pp^1_k$ has an unramified $k$-point $P$.
Then given a Galois field extension $E/k$ of group $H\subset G$, there exists a geometrically Galois extension $\mathcal E/k(T)$, with group $G$, that specializes to $E^{(G:H)}/k$ at $\pi(P) \in \Pp^1(k)$, and whose constant extension in the Galois closure is $E$.
\end{theorem}

\begin{proof}
This is a special case of Proposition~\ref{prop:constant-extension-prescribed}, with $B = \Pp^1_k$, and where we let $\mathcal E = \widetilde F^E$, using that $E/K$ is Galois with group $H$.
\end{proof}

Even more is known in the case of fields $k$ that are {\em ample} ({\em large}).  Recall that these are the fields $k$ with the property that every geometrically irreducible $k$-variety with a smooth $k$-point has infinitely many such points; they include in particular henselian fields, PAC fields, and totally real and totally $p$-adic closures of number fields.  Namely, for such fields $k$, given a finite group $G$ and a Galois extension $E/k$ with group $H \subset G$, 
there is a regular Galois extension $\mathcal E/k(T)$ with group $G$, having specified fiber $E^{(G:H)}/k$ (\cite{MR1745009}, \cite{MR1841345}); in particular, there is an affirmative answer to the arithmetic lifting problem in that situation.  Over such fields, concerning 
Question~\ref{geomBB}, the hypothesis of Theorem~\ref{geom BB thm} is satisfied and hence also the conclusion, including the assertion about the constant extension \hbox{in the Galois closure:} 

\begin{corollary} \label{BB ample}
Let $k$ be an ample (large) field, and let $G$ be a finite group with trivial center.
Given a Galois extension $E/k$ of group $H \subset G$, there is a geometrically Galois extension $F/k(T)$ of group $G$ with constant extension $E/k$ in its Galois closure, which specializes to $E^{(G:H)}/k$ at some unbranched point $t_0\in \Pp^1(k)$, and such that the corresponding cover has an unramified $k$-point.
\end{corollary}

\begin{proof}
By \cite{PopLarge} and \cite[Remark~4.3]{DeDes1}, the hypothesis on $k$ 
implies that there is
a $k$-regular Galois 
extension $N/k(T)$ of group $G$ (in fact, of any specified Galois group) with a $k$-rational point above some unbranched point $t_0\in k$. 
Thus the result follows from Theorem~\ref{geom BB thm}.
\end{proof}

Of course the key case above is with $H=G$, where we realize any given Galois extension of $k$ with group $G$ as a fiber of a geometrically Galois extension of $k(T)$.

\begin{remark} \label{rem:constant-extension}
\renewcommand{\theenumi}{\alph{enumi}}
\renewcommand{\labelenumi}{(\alph{enumi})}
\begin{enumerate}
\item \label{case_hilbertian}
Assume that $k$ is ample (large) and also Hilbertian; e.g., $k=\Qq^{\rm tr}(i)$ with $\Qq^{\rm tr}$ the field of totally real algebraic numbers. Let $H$ be a subgroup of a finite group $G$ with trivial center.  Then there exists a geometrically Galois extension $F/k(T)$ of group $G$ with constant extension $\widehat k_F/k$ of group $H$, and whose corresponding cover has an unramified $k$-point.
Indeed consider $N/k(T)$ as in the proof of Corollary~\ref{BB ample}, and similarly take a
$k$-regular Galois extension of group $H\subset G$.  Since $k$ is Hilbertian, this latter extension can be specialized to provide a Galois extension $E/k$ of group $H\subset G$. Corollary~\ref{BB ample} then provides an extension $F/k(T)$ as desired.
\item \label{constant_extension}
In general, given an extension of a field $k(T)$, the extension of constants in the Galois closure is not well understood; and this poses an obstacle for assertions such as the Regular Inverse Galois Problem.  The above results give some control over that extension of constants in special cases.  It would be desirable to obtain a more general understanding of the field extension $\widehat k_F/k$ and its Galois group.
\end{enumerate}
\end{remark}  

\begin{example} \label{example:L_0/k_not_pre-Galois} The following example completes Remark~\ref{rem:inside-Galois-closure}(\ref{inside rk b}). The extension $E/K$ constructed below is potentially Galois, of group $G$, over some extension $L/K$ 
that is pre-Galois, but the corresponding minimal extension $L_0= \LL(G)$ is not pre-Galois over $K$.

Take $K= k(T)$ with $k$ a field that, as in Remark~\ref{rem:constant-extension}(\ref{case_hilbertian}) above, is ample and Hilbertian. Fix a non-abelian simple group $G$ of order $d\geq 5$ and a nontrivial subgroup $H\subset G$, e.g. $G=A_5$ and $H =\langle (1 \hskip 2pt 2  \hskip 2pt 3  \hskip 2pt 4 \hskip 2pt 5) \rangle$. The assumptions on $k$ guarantee the existence of a $k$-regular Galois extension $N/k(T)$ of group $S_d$ whose corresponding cover has an 
unramified $k$-rational point, as well as the existence of a Galois extension $\varepsilon /k$ of group $H$ (see Remark~\ref{rem:constant-extension}(\ref{case_hilbertian})). 

Let $M\subset S_d$ be the subgroup fixing one letter and let $E$ be the fixed field in $N$ of $M$; so $E/k(T)$ is of degree $d$ and its Galois  closure is $N/k(T)$. By Proposition~\ref{hilb flds pot Gal cond}(\ref{hilb pot Gal}), $E/k(T)$ is potentially Galois of group $G$ (embedded in $S_d$ via the regular representation). By Theorem~\ref{prop:potentiallyGoG_characterization}(\ref{min fields give pot Gal}), $E/k(T)$ is in fact potentially Galois over $L_0 = \LL(G)$. Furthermore we have $L_0 = N^G$ and $N= \widehat E= EL_0$. By 
Corollary~\ref{prop:potentiallyGoG_characterization_cor}(\ref{pot cor complem}) (see also the proof of Proposition~\ref{hilb flds pot Gal cond}(\ref{potentially but no pre})), the extension $E/k(T)$ is not pre-Galois.  
In particular, $L_0/k(T)$ is not Galois. The next argument shows that $L_0/k(T)$ is not even pre-Galois.

First note that $N/k(T)$ is also the Galois closure of $L_0/k(T)$ (since $L_0\subset N$; $N/k(T)$ is Galois; 
and no nontrivial subgroup of $G$ is normal in $S_d$). By Corollary~\ref{prop:potentiallyGoG_characterization_cor}(\ref{pot cor SZ}), if  $L_0/k(T)$ were pre-Galois, then $G$ would have a normal complement in $\Gal(N/k(T))= S_d$.  But this is not the case, and so indeed $L_0/k(T)$ is not pre-Galois.

As $L_0\subset N$, the extension $L_0/k(T)$ is $k$-regular and the corresponding cover has an unramified $k$-rational 
point. One can write $L_0=k(B)$, with $B$ a smooth projective curve. The extension $N/k(B)$ is $k$-regular and is Galois of group $G$; and the corresponding cover has an unramified $k$-rational point. Thus one may apply Proposition~\ref{prop:constant-extension-prescribed}.

Let $L/k(B)$ be the extension obtained by twisting $N/k(B)$ by $\varepsilon/k$ (in the notation of Proposition \ref{prop:constant-extension-prescribed}, we have $L=\widetilde N^\varepsilon$).  
Since $G$ is simple, it follows that $N\cap L = k(B)=L_0$ and so $N/k(B)$ and 
$L/k(B)$ are linearly disjoint. By construction, $E/k(T)$ and $k(B)/k(T)$ are also linearly disjoint. Therefore $E/k(T)$ and $L/k(T)$ are linearly disjoint, and so $E/k(T)$ remains potentially Galois over $L$. Recall that, from Theorem \ref{prop:potentiallyGoG_characterization}(\ref{min subextens}), both groups $\GG(L)=\Gal(EL/L)$ and  $\GG(L_0)=\Gal(EL_0/L_0)$ coincide and are equal to $G$. Note further that the extension $L/k(T)$ is not Galois for otherwise $L_0/k(T)$ would be Galois, again by Theorem \ref{prop:potentiallyGoG_characterization}(\ref{min subextens}).

Finally we show that $L/k(T)$ is pre-Galois. More precisely, we verify below that it is pre-Galois over 
$\varepsilon(T)$. Note first that, as $\varepsilon/k$ is Galois, $\varepsilon(T)/k(T)$ is Galois too. Then, from 
Proposition~\ref{prop:constant-extension-prescribed}, we have 
$L\hskip 1pt \varepsilon(T) = L\varepsilon = N\varepsilon$.  This field is a Galois extension of $k(T)$ 
(hence a Galois extension of $\varepsilon(T)$), as it is the compositum 
of the two Galois extensions $N/k(T)$ and $\varepsilon(T)/k(T)$. Finally 
$L/k(T)$ and $\varepsilon(T)/k(T)$ are linearly disjoint since 
$L\cap \overline k = k$ (as a consequence of the $k$-regularity 
of $L/k(B)$).  Thus $L/k(T)$ is pre-Galois, as asserted.
\end{example}

\subsection{A descent result} \label{ssec:main-result} 

In this subsection we address Question~\ref{geom_Gal_base_chg} in Corollary~\ref{thm:fod}, which we deduce from a more general version, Theorem~\ref{thm: main fod}.  We also prove Corollary~\ref{cor:power-simple}, which concerns the geometric inverse Galois problem discussed in the introduction.

Recall the following useful terminology. 
Let ${\mathcal F}/k^\sep(B)$ be a finite $k^{\sep}$-regular extension.

\renewcommand{\theenumi}{\alph{enumi}}
\renewcommand{\labelenumi}{(\alph{enumi})}
\begin{enumerate}
\item
We say that ${\mathcal F}/k^\sep(B)$ 
{\em descends to $k$ as a field extension}
if there exists a $k$-regular extension $F/k(B)$ such that ${\mathcal F} = Fk^{\sep}$. If ${\mathcal F}/k^\sep(B)$ is Galois, then the extension $F/k(B)$ is geometrically Galois, of group $\Gal({\mathcal F}/k^\sep(B))=\Gal({\mathcal F\overline k}/\overline k(B))$. If in addition $F/k(B)$ is Galois, we say that ${\mathcal F}/k^\sep(B)$ {\em descends to $k$ as a Galois extension}. 

\item
Consider the subgroup $M_{\rm m}({\mathcal F}/k^\sep(B))$ of the absolute Galois group $\Gabs(k)$, consisting of all $\tau$ such that for every prolongation (equivalently for some prolongation) of $\tau$ to a $k(B)$-automorphism $\widetilde \tau$ of $k(B)^{\sep}$, there is a $k^{\sep}(B)$-isomorphism $\chi_\tau:{\mathcal F}\rightarrow {\mathcal F}^{\tilde \tau}$. The fixed field in $k^{\sep}$ of the subgroup $M_{\rm m}({\mathcal F}/k^\sep(B))$ is called  
 {\it the field of moduli of ${\mathcal F}/k^\sep(B)$ as a field extension} (relative to $k^{\sep}/k$) and is denoted by $k_{\rm m}({\mathcal F}/k^\sep(B))$, or $k_{\rm m}$ for short.

If ${\mathcal F}/k^\sep(B)$ is Galois, consider the subgroup $M_G({\mathcal F}/k^\sep(B))\subset \Gabs(k)$, consisting of all $\tau\in \Gabs(k)$ such that for every prolongation (equivalently for some prolongation) of $\tau$ to a $k(B)$-automorphism $\widetilde \tau$ of $k(B)^{\sep}$, there exists a $k^{\sep}(B)$-isomorphism $\chi_\tau:{\mathcal F}\rightarrow {\mathcal F}^{\tilde \tau}$ 
such that $\tau  \hskip 1pt \sigma  \hskip 1pt \tau^{-1} =  \chi_\tau \hskip 1pt \sigma \hskip 1pt  \chi_\tau^{-1}$ for every $\sigma \in \Gal({\mathcal F}/k^\sep(B))$.
The fixed field in $k^{\sep}$ of the subgroup $M_{\rm m}({\mathcal F}/k^\sep(B))$ is called  
 {\it the field of moduli of ${\mathcal F}/k^\sep(B)$ as a Galois extension} (relative to $k^{\sep}/k$) and is denoted by $k_G({\mathcal F}/k^\sep(B))$, or $k_G$ for short.
 
If ${\mathcal F}/k^\sep(B)$ descends to $k$ in either sense, we have $k_{\rm m} =k$ and $k_G=k$ respectively.  (For more on fields of definition and fields of moduli, see \cite{Fr77}, \cite{coombes-harbater}, \cite{DeDo1}, and \cite{DeMSRI}.)
\end{enumerate}

\begin{theorem} \label{thm: main fod} 
Assume 
that the set $B(k)$ of $k$-rational points 
of the $k$-variety $B$ is Zariski-dense.
Let ${\mathcal F}/k^\sep(B)$ be a $k^{\sep}$-regular extension of degree $d$ and with field of moduli $k$, $N/k^{\sep}(B)$ its Galois closure, $G=\Gal(N/k^{\sep}(B))$ its monodromy group (also equal to $\Gal(N\overline k/\overline k(B))$, and $G\hookrightarrow S_d$ the monodromy action associated with ${\mathcal F}/k^{\sep}(B)$. 
\renewcommand{\theenumi}{\alph{enumi}}
\renewcommand{\labelenumi}{(\alph{enumi})}
\begin{enumerate}
\item \label{fld of def a}
Then $N/k^\sep(B)$ 
descends to $k$ as a field extension, and to $k'$ as a Galois extension, for some Galois extension $k'$ of $k$ whose Galois group is 
contained in ${\rm Aut}(G)$.
\item \label{fom k_G}
Let $k_G$ be the field of moduli of $N/k^\sep(B))$ as a Galois extension.  
Then $k_G/k$ is Galois, and its Galois group $\Gal(k_G/k)$ is a subgroup of $\Nor_{S_d}(G) /(G
\Cen_{S_d}(G))$. In particular, if ${\mathcal F}/k^\sep(B)$ is Galois, then $k_G$ is its field of moduli as a Galois extension, and $\Gal(k_G/k)$ is a subgroup of $\Out(G)$.
\item \label{fld of def b}
Assume further that the exact sequence 
$1 \to Z(G) \to G \to \Inn(G) \to 1$ is split or that ${\rm cd}(k)\leq 1$. 
Then $N/k^\sep(B)$ descends to its field of moduli $k_G$ as a Galois extension.
\end{enumerate}
\end{theorem}

The fields $\Qq^{\rm ab}$ and $\overline k(T)$ are typical examples for which ${\rm cd}(k)\leq 1$.

\begin{proof}
We begin with part~(\ref{fld of def a}).  As $N$ is the compositum of all the $k^{\sep}(B)$-conjugates of ${\mathcal F}$, it follows from the field of moduli of ${\mathcal F}/k^{\sep}(T)$ being $k$ that
$N^{\tilde \tau} = N$ for every prolongation $\tilde \tau$ to $N$ of every $\tau\in \Gabs(k)$. This means that the field of moduli of $N/k^{\sep}(B)$ as a field extension is $k$. The first conclusion of~(\ref{fld of def a}) follows then from \cite[Proposition~2.5]{coombes-harbater}, 
which says that this extension descends to its field of moduli $k$.  (The cited result was originally stated with $B=\Pp^1_\Qq$, where $\overline k(B)= \overline \Qq(T)$.  But the proof extends to more general fields $k$ and varieties $B$ such that $B$ has a $k$-point 
where the cover corresponding to $N$ is not branched;
see \cite[Section~2.9 and Corollary 3.4]{DeDo1}.)

As to the second conclusion of~(\ref{fld of def a}), it rests on the standard fact (recalled at the beginning of Section \ref{ssec:geometrically_Galois})
that if a $k$-regular extension $N_0/k(B)$ induces $N/k^{\sep}(B)$, then $N_0/k(B)$ becomes Galois after extending scalars from $k$ to $\widehat k_{N_0}$ and that ${\rm Gal}(\widehat k_{N_0}/k)\subset {\rm Aut}(G)$.

For (\ref{fom k_G}), we view $k$-regular extensions of $k(B)$ as fundamental group representations. 
Recall that we have the {\it fundamental group exact sequence}
\[ 1\rightarrow \pi_1(B \smallsetminus D, t)_{k^{\hbox{\sevenrm sep}}} \rightarrow \pi_1(B \smallsetminus D,t)_k \rightarrow \Gabs(k) \rightarrow 1 \eqno{(*)} \]
\noindent
and each point $t_0\in B(k)\smallsetminus D$ provides a section ${\sf s}_{t_0}: \Gabs(k)\rightarrow \pi_1(B\smallsetminus D,t)_{k}$. 

The extension ${\mathcal F}/k^{\sep}(B)$ corresponds to a transitive homomorphism $\phi_{k^{\hbox{\sevenrm sep}}}: \pi_1(B\smallsetminus D, t)_{k^{\hbox{\sevenrm sep}}} \rightarrow S_d$ with $d=[{\mathcal F}:k^{\sep}(B)]$. 
Furthermore the Galois closure $N/k^{\sep}(B)$ corresponds to the epimorphism $\phi_{k^{\hbox{\sevenrm sep}}}: \pi_1(B\smallsetminus D, t)_{k^{\hbox{\sevenrm sep}}} \rightarrow G$ and we have $G=\phi_{k^{\hbox{\sevenrm sep}}}(\pi_1(B\smallsetminus D, t)_{k^{\hbox{\sevenrm sep}}})$.

Fix a point $t_0\in B(k)\smallsetminus D$.  
Since $k$ is the field of moduli of ${\mathcal F}/k^{\sep}(B)$, there is a natural map
\[\overline \varphi:\pi_1(B\smallsetminus D, t)_{k} \rightarrow \Nor_{S_d}(G) /\Cen_{S_d}(G)\]
such that for each $\tau \in \Gabs(k)$,
\[\phi_{k^{\hbox{\sevenrm sep}}}(x^{{\sf s}_{t_0}(\tau)}) = \phi_{k^{\hbox{\sevenrm sep}}}(x)^{(\overline \varphi\circ {\sf s}_{t_0})(\tau)} \mathrm{\ for\ all\ } x\in \pi_1(B\smallsetminus D, t)_{k^{\hbox{\sevenrm sep}}}.\]
(See \cite[Section 2.7]{DeDo1}; there this map is called the ``representation of $\pi_1(B\smallsetminus D, t)_{k}$ modulo $\Cen_{S_d}(G)$ given by the field of moduli condition''.)
Consider the subgroup 
\[H=(\overline \varphi \circ {\sf s}_{t_0})^{-1}(G\hskip 1pt \Cen_{S_d}(G)/\Cen_{S_d}(G)) \subset \Gabs(k).\]
Its fixed field in $k^{\sep}$ is the field of
moduli $k_G$ of $N/\overline k(B)$ as a Galois extension (see \cite[Section 2.7]{DeDo1}). 

As $H$ is normal in $\Gabs(k)$, the extension $k_G/k$ is Galois. More
precisely, $H$ is the kernel of the map $\overline \varphi\circ {\sf s}_{t_0}$ composed
with the canonical surjection $\Nor_{S_d}(G) / \Cen_{S_d}(G) \rightarrow
\Nor_{S_d}(G) / (G\Cen_{S_d}(G))$. This yields the desired
embedding of  $\Gal(k_G/k) \simeq \Gabs(k)/H$ into $\Nor_{S_d}(G) / G\Cen_{S_d}(G)$.
In the special case that ${\mathcal F}/k^{\sep}(B)$ is Galois, the monodromy action $G\rightarrow S_d$ is the regular representation, and the group $\Nor_{S_d}(G) / (G\Cen_{S_d}(G))$ from the general case is simply $\Out(G)$ (e.g., see the proof of \cite[Proposition 3.1]{DeDo1}).

We next turn to part~(\ref{fld of def b}). The assumption 
in this part
implies that 
the field of moduli $k_G$ of $N/k^{\sep}(B)$ is a field of definition as a Galois extension, 
by Corollary~3.2 and Corollary~3.4 (or Main Theorem~III(b)) of \cite{DeDo1}. 
\end{proof}

\begin{remark} \label{rk for main fod}
\renewcommand{\theenumi}{\alph{enumi}}
\renewcommand{\labelenumi}{(\alph{enumi})}
\begin{enumerate} 
\item \label{rk:CH needs dense}
Note that \cite[Corollary~3.4]{DeDo1} (used above) assumed that the fundamental group exact sequence $(*)$ splits; 
a condition called (Seq/Split) there.  That assumption is satisfied here, as a consequence of our assumption that $B(k)$ is Zariski-dense; so \cite[Corollary~3.4]{DeDo1} does apply. Our assumption on $B(k)$ also 
guarantees that for every branched cover of $B$ there are $k$-points of $B$ outside the branch locus; and so we can use \cite[Proposition~2.5]{coombes-harbater}.  But to be able to cite \cite[Corollary~3.4]{DeDo1}, the density assumption 
may be replaced by the weaker condition (Seq/Split).

\item \label{Hilaf} A weaker form of Theorem \ref{thm: main fod}(\ref{fom k_G}) appeared in a paper of H.~Hasson, saying that $\Gal(k_G/k)$ is a subquotient of $\Aut(G)$; see \cite[Corollary 3.8]{Has15}. 

\item \label{Nor Out}
The proof of Theorem~\ref{thm: main fod}(\ref{fom k_G}) used that $\Nor_{S_d}(G) / (G\Cen_{S_d}(G)) = \Out(G)$ if $G\hookrightarrow S_d$ is the regular representation.  More generally, for any embedding $G\hookrightarrow S_d$, there is a natural monomorphism ${\rm Nor}_{S_d}(G)/G\hskip 1pt {\rm Cen}_{S_d}(G) \hookrightarrow {\rm Out}(G)$; so $|{\rm Nor}_{S_d}(G)/G\hskip 1pt {\rm Cen}_{S_d}(G)|$ divides $|{\rm Out}(G)|$.   This is not an equality in general: e.g., if $G=A_6$ and $G \hookrightarrow S_6$ is given by the containment $A_6 \subset S_6$, then ${\rm Nor}_{S_6}(G) = S_6$ and ${\rm Cen}_{S_6}(G) = \{1\}$; so $|{\rm Nor}_{S_6}(G)/G\hskip 1pt {\rm Cen}_{S_6}(G)|= 2$ whereas $|{\rm Out}(G)|=4$.
\end{enumerate} 
\end{remark}

\begin{corollary} \label{thm:fod} 
Let $G$ be a finite group and $k$ an arbitrary field.
\renewcommand{\theenumi}{\alph{enumi}}
\renewcommand{\labelenumi}{(\alph{enumi})}
\begin{enumerate}
\item \label{geom gp monodromy}
Then $G$ is a geometric Galois group over $k$
if and only if $G$ is the monodromy group of a $k^{\rm sep}$-regular 
extension ${\mathcal F}/k^{\sep}(T)$ that has field of moduli $k$ (as a field extension). 
\item \label{geom gp wk bound}
If $G$ is a geometric Galois group over $k$, then $G$ is a regular Galois group over some Galois extension of $k$ of degree dividing $|{\rm Aut}(G)|$. 
\item \label{geom gp bound}
Let $G$ be a geometric Galois group over $k$, and suppose that
the exact sequence $1 \to Z(G) \to G \to \Inn(G) \to 1$ is split or that ${\rm cd}(k)\leq 1$.  Then $G$ is a regular Galois group over some Galois extension of $k$ of degree dividing $|{\rm Out}(G)|$. Consequently, if in addition $\Out(G)$ is trivial, then 
$G$ is a regular Galois group over $k$.
\end{enumerate}
\end{corollary}

\begin{proof} The condition that the group $G$ is a geometric Galois group over $k$ means that there is a $k$-regular extension $F/k(T)$ such that $F\overline k/\overline k(T)$ is Galois of group $G$; or equivalently, as remarked in Section \ref{ssec:geometrically_Galois}, such that  $Fk^\sep/k^\sep(T)$ is Galois of group $G$.

For the forward implication in part~(\ref{geom gp monodromy}),  the asserted conclusion holds with $\mathcal F = F k^{\rm sep}$. The reverse implication in part~(\ref{geom gp monodromy}) and the assertion in part~(\ref{geom gp wk bound}) follow from Theorem \ref{thm: main fod}(\ref{fld of def a}).
We obtain part (\ref{geom gp bound}) from 
Theorem~\ref{thm: main fod}(\ref{fom k_G},\ref{fld of def b}).
\end{proof}

\begin{remark}\label{GeoIGP=RIGP}
\renewcommand{\theenumi}{\alph{enumi}}
\renewcommand{\labelenumi}{(\alph{enumi})}
\begin{enumerate}
\item \label{RIGPandGeoIGP}
Corollary \ref{thm:fod} has the following consequence, which can be compared to Corollary \ref{preIGP=IGP}: given a  field $k$, the statements that all finite groups are geometric Galois groups over $k$ and that all finite groups are regular Galois groups over $k$ are equivalent. 
Namely, let $G$ be a finite group. By \cite[Theorem~1]{HR80}, $G$ is a quotient of some
complete group $\widetilde G$. If all finite groups are geometric Galois groups over $k$, then $\widetilde G$ is.
By Corollary~\ref{thm:fod}, $\widetilde G$ is a regular Galois group over $k$. It follows that $G$ is a regular Galois group over $k$ as well.
\item \label{RIGPandGeoIGP_Hurwitz} Remark~(\ref{RIGPandGeoIGP}) above 
suggests a possible strategy for attacking the RIGP.  
Given a finite group $G$ and an integer $r \ge 1$, there is a moduli space in characteristic zero for the Galois branched covers of $\Pp^1$ with $r$ branch points whose Galois group is isomorphic to $G$.  (E.g., see \cite[Section~1]{coombes-harbater}, where this {\it Hurwitz space} is denoted by $\widetilde{\mathcal P}$; and \cite[Section~1.2]{FV91}, where it is denoted by ${\mathcal H}^{\rm ab}_r(G)$.)  For $k$ an extension of $\Qq$, a $k$-rational point on this space corresponds to a Galois extension of $\overline k(T)$ with group $G$ whose field of moduli as a field extension is contained in $k$; or equivalently (by Theorem~\ref{thm: main fod}(\ref{fld of def a})) to a geometrically Galois extension of $k(T)$ with group $G$.  If for {\it every} finite group $G$ there is a $k$-point on ${\mathcal H}^{\rm ab}_r(G)$ for some $r$ depending on $G$, then Remark~(\ref{RIGPandGeoIGP}) implies that every finite group is a regular Galois group over $k$.  
(Compare this with \cite[Theorem~1]{FV91}, which considers a related Hurwitz space ${\mathcal H}^{\rm in}_r(G)$, parametrizing pairs consisting of a cover as above together with an isomorphism of its Galois group with $G$.  It shows that a point on ${\mathcal H}^{\rm in}_r(G)$ 
corresponds to a Galois extension of $\overline k(T)$ with group $G$ whose field of moduli as a Galois extension contains $k$; and hence to a Galois extension of $k(T)$ with group $G$ if $G$ has trivial center.)
\end{enumerate}
\end{remark}

\begin{corollary} \label{cor:power-simple}  Let $k$ be a  field and $G$ be a finite group.
\renewcommand{\theenumi}{\alph{enumi}}
\renewcommand{\labelenumi}{(\alph{enumi})}
\begin{enumerate}
\item \label{quot of geom Gal}
Then $G$ is a quotient of some geometric Galois group over $k$. Hence if $k$ is Hilbertian, $G$ is also a quotient 
of a pre-Galois group over $k$.
\item \label{part 2}
If $G$ is a simple group, then some power $G^n$ is a geometric 
 Galois group over $k$. Hence if $k$ is Hilbertian, $G^n$ is a pre-Galois group over $k$. 
\end{enumerate}
\end{corollary}

\begin{proof}
The RIGP property holds over the field $k^{\sep}$ by \cite{PopLarge}, since that field is ample (large); if $k$ is perfect it also follows from \cite[Corollary~1.5]{Ha1} since then $k^{\sep} = \overline k$. Thus there exists a $k^{\sep}$-regular Galois extension $F/k^{\sep}(T)$ of group $G$.
Consider the algebraic extension $F/k(T)$ and denote its
Galois closure by ${\mathcal F}/k(T)$. The extension ${\mathcal F}/k^{\sep}(T)$ is finite and Galois, say of group ${\mathcal G}$. As $F/k^{\sep}(T)$ is Galois, $G$ is a quotient of ${\mathcal G}$. Furthermore, as ${\mathcal F}/k(T)$ is Galois, the field of moduli of ${\mathcal F}/k^{\sep}(T)$ as a field extension is $k$. 

For the first assertion of part~(\ref{quot of geom Gal}), note that Theorem~\ref{thm: main fod}(\ref{fld of def a}) yields 
that ${\mathcal F}/k^{\sep}(T)$ 
descends to a field extension
${\mathcal F}_0/k(T)$ defined over $k$. 
By construction ${\mathcal F}_0/k(T)$ is geometrically Galois of group ${\mathcal G}$ and 
$G$ is a quotient of ${\mathcal G}$.

For the first assertion of (\ref{part 2}), assume that $G$ is simple. We may assume that $G$ is non-abelian, since every cyclic group of prime order is the Galois group of a $k$-regular Galois extension of $k(T)$.  

Since $F$ is defined over $E$ for some finite extension $E/k$, there are finitely many distinct conjugate fields $F^\sigma$ of $F$ over $k(T)$, as $\sigma$ ranges over $\Gal(k)$.  The field ${\mathcal F}$ is the compositum of these fields $F^\sigma$ in $\overline{k(T)}$.  

We claim that ${\mathcal G}=\Gal({\mathcal F}/k^{\sep}(T))$ is of the form $G^n$.  This follows by induction from the following elementary assertion: If $L_1,L_2$ are Galois field extensions of a field $K$ with Galois groups $G_1,G_2$ such that $G_1$ is simple, and if $L_1$ is not contained in $L_2$, then $\Gal(L_1L_2/K)=G_1 \times G_2$.  (This last assertion holds because the simplicity of $G_1$ implies that $L_1 \cap L_2 = K$).  

By Corollary~\ref{thm:fod}(\ref{geom gp monodromy}), it then follows that $G^n$ is a geometric 
Galois group over $k$.  

The last assertions in parts (\ref{quot of geom Gal}) and (\ref{part 2}), for $k$ Hilbertian, follow by
Proposition \ref{intro:IGP-implications}.
\end{proof}

\subsection{Extensions ${\mathcal F}/\overline k(T)$ with field of moduli $k$} \label{ssec:ext-fom} 
The notion of rigidity has been used to realize many finite groups as Galois groups over $\Qq$, or over small extensions of $\Qq$.  In this subsection we generalize that notion and apply Theorem~\ref{thm: main fod} in order to obtain sharper results along those lines, and to obtain results about geometric Galois groups and pre-Galois groups, in 
Theorem~\ref{thm:main wk rig}.
Here we work over subfields $k$ of $\Cc$, so that we can rely on the correspondence between branched covers and tuples of elements, given by Riemann's Existence Theorem.  

\begin{definition} \label{def:weak-rigid} 
Let ${\bf C}=(C_1,\ldots,C_r)$ be a tuple of conjugacy classes of a finite group $G$.  We say that ${\bf C}$ is {\it weakly rigid with respect to an embedding $G\hookrightarrow S_d$} if 
\renewcommand{\theenumi}{\alph{enumi}}
\renewcommand{\labelenumi}{(\alph{enumi})}
\begin{enumerate}
\item \label{wk rig a}
there are generators $g_1,\ldots,g_r$ of $G$ with $g_1\cdots g_r = 1$ and $g_i\in C_i$, $i=1,\ldots,r$, and
\item \label{wk rig b}
if $g^\prime_1,\ldots,g^\prime_r$ are generators of $G$ with the same properties, there exists $\omega \in S_d$ such that $g^\prime_i = \omega g_i \omega^{-1}$, $i=1,\ldots,r$.
\end{enumerate}
\end{definition}

In the special case where $G\hookrightarrow S_d$ is the regular representation, this is equivalent to the 
traditional notion of being {\em weakly rigid}, in which the conclusion of part~(\ref{wk rig b}) is replaced by the condition that for each choice of $g'_1,\ldots,g'_r$ there exists $\sigma \in \Aut(G)$ such that each $\sigma(g_i) = g_i'$.  (See \cite[Def.~2.15]{Vo96}.  If in addition $\sigma$ is an inner automorphism of $G$, i.e.\ if $\omega$ is in the image of $G$, then the tuple is {\em rigid}.)    
When that narrower condition of weak rigidity is satisfied, 
Theorem~\ref{thm: main fod} has the following consequence:

\begin{corollary} \label{cor: wk rig} 
Let $G$ be a finite group with a weakly rigid tuple of conjugacy classes.  Then
\renewcommand{\theenumi}{\alph{enumi}}
\renewcommand{\labelenumi}{(\alph{enumi})}
\begin{enumerate}
\item \label{geom Gal over abel}
$G$ is a geometric Galois group over $\Qq^{\rm ab}$ and a regular Galois group over some Galois extension of $\Qq^{\rm ab}$ of degree dividing $|{\rm Out}(G)|$.
\item \label{pre Gal over abel}
$G$ is a pre-Galois group over $\Qq^{\rm ab}$ and is a Galois group over some Galois extension of $\Qq^{\rm ab}$ of degree dividing $|{\rm Out}(G)|$.
\end{enumerate}
\end{corollary}

\begin{proof}
Since the given tuple of conjugacy classes ${\bf C}=(C_1,\ldots,C_r)$ is weakly rigid, \cite[Theorem~2.17]{Vo96} applies.  That is, up to isomorphism of field extensions, there is a unique finite extension $L/\Cc(T)$ that is Galois with group $G$ and that corresponds to a branched cover $X \to \Pp^1_{\overline\Qq}$, with given rational branch points, and having branch cycle description $(g_1,\dots,g_r)$ for some $g_i \in C_i$.  Since the branch points are rational, this branched cover descends to $\overline\Qq(T)$.  Moreover the field of moduli (relative to $\overline\Qq/\Qq^{\rm ab}$) of the corresponding field extension of $\overline\Qq(T)$ is equal to $\Qq^{\rm ab}$, since $\Gal(\Qq^{\rm ab})$ acts on branch cycle descriptions by preserving conjugacy classes (cf.~\cite[Lemma~2.8]{Vo96}).  

The first assertion in part~(\ref{geom Gal over abel}) of the corollary now follows from 
Theorem~\ref{thm: main fod}(\ref{fld of def a}).  The second assertion in part~(\ref{geom Gal over abel}) follows from Theorem~\ref{thm: main fod}(\ref{fom k_G}, \ref{fld of def b}), using that $\cd(\Qq^{\rm ab})=1$. Part~(\ref{pre Gal over abel}) of the corollary follows from part~(\ref{geom Gal over abel}) by specialization in the Hilbertian field $\Qq^{\rm ab}$ (see Proposition~\ref{intro:IGP-implications}) in combination with Proposition~\ref{prop:geom vs pot}(\ref{geom Gal pre}). 
\end{proof}

In Theorem~\ref{thm:main wk rig} below, we prove a more general result using the broader notion given in Definition~\ref{def:weak-rigid}.
First we extend the usual notion of {\em rationality} (see \cite[Definition~3.7]{Vo96}) to that more general setting.

\begin{definition} \label{def:C-k-rational} The $r$-tuple ${\bf C}$ 
is {\it weakly rational with respect to an embedding $G\hookrightarrow S_d$} if  for each $m\in (\Zz/d\Zz)^\times$, $C_1^m,\ldots,C_r^m$ is a permutation of $C_1,\ldots,C_r$, up to conjugation by an element $\omega\in {\rm Nor}_{S_d}(G)$. Given a subfield $k\subset \Cc$, ${\bf C}$ is {\it weakly $k$-rational with respect to $G\hookrightarrow S_d$} if the condition holds for values $m$ of the cyclotomic character $\chi_k: \Gabs(k) \rightarrow (\Zz/d\Zz)^\times$. 
\end{definition}

For short, we just say {\it weakly rigid}, 
{\it weakly rational} and {\it weakly $k$-rational} if $G\hookrightarrow S_d$ is the regular representation of $G$; in this case, conjugating by some element $\omega \in {\rm Nor}_{S_d}(G)$ is equivalent to acting by some automorphism $\gamma \in {\rm Aut}(G)$. The conditions are then weaker than with any other embedding $G\hookrightarrow S_d$. The classical  
rational and $k$-rational notions correspond to the special situation where one can take
$\omega \in S_d$ to be in (the image of) $G$, in the definition above.

Recall the classical action of $\Gabs(k)$ on the conjugacy classes $C$ of $G$: for each $\tau \in \Gabs(k)$, $C$ is mapped to $C^{1/\chi_k(\tau)}$ where $\chi_k:\Gabs(k) \rightarrow (\Zz/d\Zz)^\times$ is the cyclotomic character modulo $d$: that is $\zeta_d^\tau = \zeta_d^{\chi_k(\tau)}$ where $\zeta_d=e^{2i\pi/d}$. 

We use the notation ${\bf C}$ for a $r$-tuple $(C_1,\ldots,C_r)$ of conjugacy classes of $G$ and ${\bf t}$ for a $r$-tuple $(t_1,\ldots,t_r)$ of distinct points in $\Pp^1(\overline k)$. We always assume that the sets $\{C_1,\ldots,C_r\}$ and $\{t_1,\ldots,t_r\}$ are invariant under the action of $\Gabs(k)$. For $\tau \in \Gabs(k)$ and $i=1,\ldots,r$, denote by $\tau(i)$ the index such that $t_i^\tau = t_{\tau(i)}$.

\begin{definition} \label{def:weak-rat-ram-type}
A triple $[G\hookrightarrow S_d ,{\bf C}, {\bf t}]$ is {\it weakly $k$-rational} if
for each $\tau \in \Gabs(k)$, there exists some element $\omega_\tau \in {\rm Nor}_{S_d}(G)$ such that 
\[C_{\tau(i)}^{\chi_k(\tau)} = C_i^{\omega_\tau}, \ \ i=1,\ldots,r.\]
It is {\it $k$-rational} if one can take $\omega_\tau=1$ ($\tau\in \Gabs(k)$).
\end{definition}

Definition \ref{def:weak-rat-ram-type} generalizes definitions from \cite{Vo96} of {\it (weakly) $k$-rational type}  for which  $G\hookrightarrow S_d$ is the regular representation: see Definition~3.7 and Remark~3.9(b) there.

If $[G\hookrightarrow S_d, {\bf C}, {\bf t}]$ is weakly $k$-rational (\hbox{resp.} is $k$-rational), 
then in particular the $r$-tuple ${\bf C}$ is weakly $k$-rational (\hbox{resp.} is $k$-rational) with respect to the 
embedding $G\hookrightarrow S_d$.  The former condition is stronger in that the permutation of $C_1,...,C_r$ involved in 
Definition \ref{def:weak-rat-ram-type}
 should be the permutation induced by the action of $\tau$ on the branch points, for any $\tau \in \Gabs(k)$. 

Triples $[G\hookrightarrow S_d, {\bf C}, {\bf t}]$ are used to represent the {\it ramification invariants} of an extension 
 ${\mathcal F}/\overline k(T)$: $G$ is the monodromy group, $G\hookrightarrow S_d$ the monodromy action, ${\bf t}$ the branch point tuple and ${\bf C}$ the tuple of inertia canonical classes (in the Galois closure), with ${\bf t}$ and ${\bf C}$ ordered in such a way that $C_i$ corresponds to $t_i$, $i=1,\ldots,r$.

\smallskip

The following lemmas adjust classical rigidity statements to the non-Galois situation.

\begin{lemma} \label{thm:weakl-rigid-criterion} Let  ${\mathcal F}/\overline k(T)$ be an extension such that its ramification invariant $[G\hookrightarrow S_d, {\bf C}, {\bf t}]$ is weakly $k$-rational and ${\bf C}$ is weakly rigid \hbox{\it w.r.t.} $G\hookrightarrow S_d$. 
Then ${\mathcal F}/\overline k(T)$ has field of moduli $k$ as a field extension.
\end{lemma}

\begin{proof}
In the case $G\hookrightarrow S_d$ is the regular representation, a proof is given in Theorem~3.8 
and Remark~3.9(b)-(c) in \cite{Vo96}. 
We consider a more general embedding $G\hookrightarrow S_d$. 

Every $\tau\in \Gabs(k)$ maps the extension ${\mathcal F}/\overline k(T)$ to some conjugate extension ${\mathcal F}^\tau/\overline k(T)$, which is of 
ramification invariant  
\[[\, G\hookrightarrow S_d, \hskip 2pt (C_{1}^{1/\chi_k(\tau)},\ldots, C_{r}^{1/\chi_k(\tau)}), \hskip 2pt (t_1^\tau,\ldots, t_r^\tau))\,];\]
or, after reordering $(t_1^\tau,\ldots, t_r^\tau)$ in $(t_1,\ldots, t_r)$,
\[[\, G\hookrightarrow S_d, \hskip 2pt (C_{\tau(1)}^{\chi_k(\tau)},\ldots, C_{\tau(r)}^{\chi_k(\tau)}), \hskip 2pt (t_1,\ldots, t_r))\,].\]
Due to the weak $k$-rationality assumption, this triple is $[G\hookrightarrow S_d, {\bf C}^{\omega_\tau}, {\bf t}]$ for some $\omega_\tau\in {\rm Nor}_{S_d}(G)$.  By the weak rigidity assumption, there is a unique isomorphism class of extensions of $\overline k(T)$ with this ramification invariant. Therefore ${\mathcal F}^\tau/\overline k(T)$ is $\overline k(T)$-conjugate to ${\mathcal F}/\overline k(T)$. As this holds for every $\tau \in \Gabs(k)$, the field of moduli of ${\mathcal F}/\overline k(T)$ is $k$.
\end{proof}

\begin{lemma} \label{lem:construct-t} Let ${\bf C}$ be an $r$-tuple of conjugacy classes of $G$. If ${\bf C}$ is weakly $k$-rational \hbox{\it w.r.t.} an embedding $G\hookrightarrow S_d$, then there exists an $r$-tuple ${\bf t}\subset \Pp^1(\overline k)$ such that the triple $[G\hookrightarrow S_d, {\bf C}, {\bf t}]$ is weakly $k$-rational.
\end{lemma} 

\begin{proof}
The construction is explained in Lemma 3.16 of \cite{Vo96} in the case $G\hookrightarrow S_d$ is the regular representation and easily generalizes to our situation: the only change is that the conjugacy classes should be regarded modulo the action of ${\rm Nor}_{S_d}(G)$.
\end{proof}

\begin{theorem} \label{thm:main wk rig}
Let $G$ be a finite group and $k$ be a subfield of $\Cc$.
\renewcommand{\theenumi}{\alph{enumi}}
\renewcommand{\labelenumi}{(\alph{enumi})}
\begin{enumerate}
\item \label{wk rig wk rat}
Assume $G$ has a weakly rigid tuple ${\bf C}$ that is also weakly $k$-rational. 
Then $G$ is a geometric Galois group over $k$, and is 
a regular Galois group over 
for some Galois extension $\widehat k$ of $k$ of degree 
$[\widehat k:k]$ dividing $|{\rm Aut}(G)|$. 
\item \label{wk rig trans emb}
Assume $G$ has a tuple ${\bf C}$ that is weakly rigid with respect to some transitive embedding $G\hookrightarrow S_d$ and is weakly $k$-rational for some field $k$ with respect to the same embedding. Suppose that 
the exact sequence 
$1 \to Z(G) \to G \to \Inn(G) \to 1$ is split
or that $k$ is of cohomological dimension ${\rm cd}(k)\leq 1$.  Then 
$G$ is a regular Galois group over some Galois extension $k_G$ of $k$ of degree dividing
$|{\rm Nor}_{S_d}(G)/G\hskip 1pt {\rm Cen}_{S_d}(G)|$ and hence $|{\rm Out}(G)|$.
\end{enumerate}
\end{theorem}

\begin{proof} 
We will show that the assumptions guarantee that there is an extension ${\mathcal F}/\overline k(T)$ with field of moduli $k$ and then apply Theorem \ref{thm: main fod}. 
Fix a finite group $G$, a transitive embedding $G\hookrightarrow S_d$, an integer $r\geq 2$ 
and a subfield $k\subset \Cc$. 

Assume as in part~(\ref{wk rig trans emb}) that $G$ is given with an $r$-tuple ${\bf C}$ of conjugacy classes that is weakly $k$-rational and weakly rigid with respect to an embedding $G\hookrightarrow S_d$.  By Lemma \ref{lem:construct-t}, there is an $r$-tuple ${\bf t}\subset \Pp^1(\overline k)$ such that $[G\hookrightarrow S_d, {\bf C}, {\bf t}]$ is weakly $k$-rational. Consider next an extension ${\mathcal F}/\overline{k}(T)$ of degree $d$, of monodromy group $G\hookrightarrow S_d$, of branch point set ${\bf t}$, and corresponding canonical inertia invariant ${\bf C}$; the existence of such an extension is guaranteed by the Riemann Existence Theorem.  It follows from Lemma \ref{thm:weakl-rigid-criterion} that ${\mathcal F}/{\overline k}(T)$ has
field of moduli $k$ as a field extension.  By Theorem~\ref{thm: main fod}(\ref{fld of def b}), if $N/\overline k(T)$ is the Galois closure of  ${\mathcal F}/{\overline k}(T)$, then $N/\overline k(T)$ descends to its field of moduli $k_G$ as a Galois extension; this is a regular realization of $G$, as asserted in part~(\ref{wk rig trans emb}).
The assertion about the degree then follows from Theorem~\ref{thm: main fod}(\ref{fom k_G}) and Remark~\ref{rk for main fod}(\ref{Nor Out}).

For part~(\ref{wk rig wk rat}), we simply note that ``weakly rigid and weakly $k$-rational'' 
is the same as ``weakly rigid and weakly $k$-rational with respect to the regular representation $G\hookrightarrow S_{|G|}$''.
We can then proceed as above but use Theorem~\ref{thm: main fod}(\ref{fld of def a}) instead of Theorem~\ref{thm: main fod}(\ref{fom k_G},\ref{fld of def b}). 
\end{proof}

Note that Corollary~\ref{cor: wk rig}(\ref{geom Gal over abel}) is the special case of Theorem~\ref{thm:main wk rig} for which $k=\Qq^{\rm ab}$ and the embedding $G\hookrightarrow S_d$ in (\ref{wk rig trans emb}) is given by the regular representation.  Again this uses that ${\rm cd}(\Qq^{\rm ab}) = 1$, along with the fact that $\Gal(\Qq^{\rm ab})$ acts on branch cycle descriptions by preserving conjugacy classes (i.e., ${\bf C}$ is $\Qq^{\rm ab}$-rational).


\end{document}